\documentclass{article}

\usepackage[preprint,nonatbib]{neurips_2025}

\usepackage[utf8]{inputenc} %
\usepackage[T1]{fontenc}    %
\usepackage{url}            %
\usepackage{microtype}      %

\usepackage{textcomp}
\usepackage{lmodern}
\usepackage[english]{babel}
\usepackage{lipsum}
\usepackage{graphicx}
\usepackage{adjustbox}
\usepackage{scalerel,stackengine,amsmath}
\usepackage{subcaption}
\usepackage{tikz}
\usetikzlibrary{positioning, calc}
\usepackage{datetime}

\usepackage[ruled,linesnumbered]{algorithm2e}
\SetKwComment{Comment}{$\triangleright$ }{}
\RestyleAlgo{ruled}
\usepackage{accents}
\usepackage{xcolor,color}
\usepackage{parskip}
\usepackage{amssymb,amsmath,amsthm,amsfonts,mathtools,mathtools, mathrsfs,bbold}
\usepackage[colorlinks=true, urlcolor=blue, linktoc=all]{hyperref}
\usepackage{thmtools}
\usepackage{stmaryrd}
\usepackage{multirow}
\usepackage{booktabs}
\usepackage{alphabeta} %
\usepackage{thm-restate}
\usepackage{nicefrac}
\usepackage[capitalise,nameinlink,noabbrev]{cleveref}

\newcommand*\circledaux[1]{\tikz[baseline=(char.base)]{
    \node[shape=circle,draw,inner sep=0.8pt] (char) {#1};}}

\NewDocumentCommand{\circled}{ m o }{%
    \IfNoValueTF{#2}{ \circledaux{#1} }{ \stackrel{\circledaux{#1}}{#2} }%
}

\newcommand{\R}{\mathbb{R}}
\newcommand{\ones}{\mathbb{1}}

\makeatletter

\def\Hy@raisedlink@left#1{%
    \ifvmode
        #1%
    \else
        \Hy@SaveSpaceFactor
        \llap{\smash{%
        \begingroup
            \let\HyperRaiseLinkLength\@tempdima
            \setlength\HyperRaiseLinkLength\HyperRaiseLinkDefault
            \HyperRaiseLinkHook
        \expandafter\endgroup
        \expandafter\raise\the\HyperRaiseLinkLength\hbox{%
            \Hy@RestoreSpaceFactor
            #1%
            \Hy@SaveSpaceFactor
        }%
        }}%
        \Hy@RestoreSpaceFactor
        \penalty\@M\hskip\z@
    \fi
}
\makeatother

\newcommand\linktoproof[1]{{\normalfont[{\hyperlink{proof:#1}{$\downarrow$}}]}}
\newcommand\linkofproof[1]{\textbf{of \cref{#1}. }\newtarget{proof:#1}}

\newcommand{\bigomega}[1]{\Omega( #1 )}
\newcommand{\bigomegal}[1]{\Omega\left( #1 \right)}

\newcommand{\bigol}[1]{O\left( #1 \right)}

\newcommand{\jmlrBlackBox}{\rule{1.5ex}{1.5ex}}

\newcommand{\jmlrQED}{\hfill\jmlrBlackBox\par\bigskip}
\providecommand{\proofname}{Proof}
\renewenvironment{proof}%
{%
 \par\noindent{\bfseries\upshape \proofname\ }%
}%
{\jmlrQED}

\crefname{equation}{Eq.}{Eqs.}
\newcommand{\mP}{\mathcal{P}}
\newcommand{\mQ}{\mathcal{Q}}
\newcommand{\mA}{\mathcal{A}}

\newcommand{\mX}{\mathcal{X}}
\newcommand{\verts}[1]{\mathcal{V}_{#1}}
\newcommand{\conv}[1]{\mathsf{conv}\left(#1\right)}
\newcommand{\aff}[1]{\mathsf{aff}\left(#1\right)}
\newcommand{\faces}[1]{{\mathsf{faces}^\ast}\left(#1\right)}
\newcommand{\facets}[1]{\mathsf{facets}\left(#1\right)}
\newcommand{\dist}[1]{\mathsf{dist}\left(#1\right)}

\NewDocumentCommand{\pw}{o m}{%
  \mathop{%
    \delta%
    \IfValueT{#1}{^{#1}}%
    _{#2}%
  }%
}
\NewDocumentCommand{\vf}{o m}{%
  \mathsf{vf}%
  \IfValueT{#1}{^{#1}}%
  _{#2}%
}
\NewDocumentCommand{\diam}{o m}{
  D
  \IfValueT{#1}{^{#1}}_#2
}
\newcommand{\innp}[1]{\left\langle #1 \right\rangle}

\DeclareMathOperator*{\argmin}{\mathrm{arg\,min}}
\DeclareMathOperator*{\argmax}{\mathrm{arg\,max}}
\DeclarePairedDelimiter\ceil{\lceil}{\rceil}
\DeclarePairedDelimiter\floor{\lfloor}{\rfloor}
\newcommand{\norm}[1]{\| #1 \|}
\newcommand{\defi}{\stackrel{\mathrm{\scriptscriptstyle def}}{=}}

\newtheorem{theorem}{Theorem}[section]
\newtheorem{lemma}[theorem]{Lemma}

\newtheorem{proposition}[theorem]{Proposition}
\newtheorem{definition}[theorem]{Definition}
\newtheorem{corollary}[theorem]{Corollary}

\hypersetup{
    colorlinks=true,
    linkcolor=blue,
    filecolor=magenta,
    urlcolor=blue,
    citecolor=blue
}

\title{Linear Convergence of the Frank-Wolfe Algorithm over Product Polytopes}

\author{%
    Gabriele Iommazzo\\
    Zuse Institute Berlin, Berlin, Germany\\
    \texttt{iommazzo@zib.de}\\
    \And
    David Martínez-Rubio\\
    Signal Theory\\ and Communications Department,\\ Carlos III University,\\ Madrid, Spain\\
    \texttt{dmrubio@ing.uc3m.es}\\
    \And
    Francisco Criado \\
    Departamento de Matemáticas,\\ CUNEF Universidad, Madrid, Spain \\
    \texttt{francisco.criado@cunef.edu}\\
    \And
    Elias Wirth\\
    Berlin Institute of Technology,\\ Berlin, Germany \\
    \texttt{wirth@math.tu-berlin.de}\\
    \And
    Sebastian Pokutta\\
    Technische Universität Berlin,\\ and Zuse Institute Berlin,\\ Berlin, Germany \\
    \texttt{pokutta@zib.de}\\
}

\let\epsilon\varepsilon

\usepackage[natbib, backend=biber, maxcitenames=3, minalphanames=3, maxbibnames=99, style=alphabetic, hyperref, backref, useprefix=true, uniquename=false, doi=false,url=false,eprint=false]{biblatex} 

\usepackage{csquotes}               
\bibliography{refs}        

\newbibmacro{string+doiurlisbn}[1]{%
  \iffieldundef{doi}{%
    \iffieldundef{url}{%
      \iffieldundef{isbn}{%
        \iffieldundef{issn}{%
          #1%
        }{%
          \href{http://books.google.com/books?vid=ISSN\thefield{issn}}{#1}%
        }%
      }{%
        \href{http://books.google.com/books?vid=ISBN\thefield{isbn}}{#1}%
      }%
    }{%
      \href{\thefield{url}}{#1}%
    }%
  }{%
    \href{https://doi.org/\thefield{doi}}{#1}%
  }%
}

\DeclareFieldFormat{title}{\usebibmacro{string+doiurlisbn}{\mkbibemph{#1}}}
\DeclareFieldFormat[article,incollection,inproceedings]{title}%
    {\usebibmacro{string+doiurlisbn}{#1}}

\usepackage{todonotes}
\newif\ifTODO 

\TODOfalse

\definecolor{aquamarine}{rgb}{0.5, 1.0, 0.83}

\newcommand{\Dav}[1]{\ifTODO \todo[inline,author=DM, caption={}, color=yellow]{#1} \fi}

\newcommand{\dav}[1]{\ifTODO \todo[size=\scriptsize, color=yellow]{DM: #1}{} \fi}

\newcommand{\eli}[1]{\ifTODO \todo[size=\scriptsize, color=magenta]{EW: #1}{} \fi}

\newcommand{\gabrev}[1]{\textcolor{black}{#1}}

\ifTODO
\paperwidth=\dimexpr \paperwidth + 6cm\relax
\oddsidemargin=\dimexpr\oddsidemargin + 3cm\relax
\evensidemargin=\dimexpr\evensidemargin + 3cm\relax
\marginparwidth=\dimexpr \marginparwidth + 3cm\relax
\fi

\usepackage{xparse}
\usepackage{xargs} 

\newcommand\newlink[2]{\hyperlink{#1}{\normalcolor #2}}
\makeatletter
\newcommand\newtarget[2]{\Hy@raisedlink{\hypertarget{#1}{}}#2}
\makeatother

\newcommand{\PL}{\newlink{def:acronym_polyak_lojasiewicz}{PL}}
\newcommand{\QG}{\newlink{def:acronym_quadratic_growth}{QG}}
\newcommand{\FW}{\newlink{def:acronym_frank_wolfe}{FW}}
\newcommand{\AFW}{\newlink{def:acronym_away_frank_wolfe}{AFW}}
\newcommand{\LMO}{\newlink{def:acronym_linear_minimization_oracle}{LMO}}

\newcommand{\w}[1]{\operatorname{\delta}_{#1}}
\newcommand{\what}{\operatorname{\bar{\delta}}}
\newcommand{\X}{\mathcal{X}}

\newcommand{\wface}[2]{\operatorname{\Delta}_{#1}(#2)}
\newcommand{\whatface}[2]{\operatorname{\bar{\Delta}}_{#1}(#2)}

\begin{document}

\maketitle
\begin{abstract}
We study the linear convergence of Frank-Wolfe algorithms over product polytopes. We analyze two condition numbers for the product polytope, namely the \emph{pyramidal width} and the \emph{vertex-facet distance}, based on the condition numbers of individual polytope components.
%
As a result, for convex objectives that are $\mu$-Polyak-\L{}ojasiewicz, we show linear convergence rates quantified in terms of the resulting condition numbers.
We apply our results to the problem of approximately finding a feasible point in a polytope intersection in high-dimensions, and demonstrate the practical efficiency of our algorithms through empirical results.

\end{abstract}

\section{Introduction}\label{s:intro}

Frank-Wolfe algorithms (\newtarget{def:acronym_frank_wolfe}{\FW{}}), also known as Conditional Gradient methods, are widely used first-order optimization algorithms. They are particularly suited for constrained convex problems $\min_{x\in\X} f(x)$, for a convex compact set $\X$, where linear minimization over $\X$ is used to keep the iterates feasible. The simplicity of this family of methods, along with its projection-free nature, make it especially appealing in large-scale settings, where projections can be costly. This is often the case when the feasible set is a polytope \cite{BCC+23}. More concretely, \FW{} methods compute iterates as combinations of extreme points in the feasible set $\mX$ returned by a linear minimization oracle (\newtarget{def:acronym_linear_minimization_oracle}{\LMO{}}):
\begin{equation}
g \mapsto \argmin_{x \in \mX} \innp{g, x} \tag{LMO}
\label{eq:lmo}
\end{equation}
on $\mX$, where $g$ is some vector, usually the current gradient of the objective function or an approximation thereof.
After computing \eqref{eq:lmo}, the standard \FW{} method takes a step from the current iterate in the direction of the \LMO{} solution.
In many scenarios, \eqref{eq:lmo}
is a cheaper alternative to the quadratic optimization problem that projections consist of, cf. \cite{Jag13}, and can be computed even when $\mX$ is defined implicitly, such as a convex hull of a set of points.

The convergence behavior of Frank-Wolfe methods is strongly influenced by the geometry of the constraint set, that of the function, and the location of the unconstrained minimizer. While sublinear rates are typical in general convex and smooth settings, it is known that linear convergence can be achieved under additional assumptions, such as global optimizers being in the interior of $\X$ \citep{wolfe1970convergence,guelat1986some}, strong convexity of the feasible set along with all global optimizers being outside of it \citep{levitin1966constrained}, or strong convexity of $f$ for $\X$ being a polytope having some positive condition number. This last assumption holds for one of several different notions of conditioning that have been defined \citep{lacoste2013affine,GH15,LJ15,BS15,penya2016vonneumann,PR18}, 
among which there is the pyramidal width and the vertex-facet distance that we use in this paper, cf. \cref{s:premises}. A few works extend the applicability of the algorithms to larger classes of functions such as the ones satisfying the Polyak-\L{}ojasiewicz (\newtarget{def:acronym_polyak_lojasiewicz}{\PL{}}) condition, as opposed to strong convexity \citep{BS15,BCC+23}. We quantify the implications of our condition number results in terms of linear convergence rates for our applications.

Importantly, we note that much of the aforementioned theory remains confined to the setting of single polytopes, with limited understanding of how favorable properties like good condition numbers might extend and compose in more structured domains. In this work, we focus on optimization over \emph{product polytopes}. This setting arises naturally in machine learning applications, specially via the feasibility problem for the intersection of several feasible sets, such as in image representation in computerized tomography \cite{CCC+09}, compressed sensing \cite{CCP12},
and including signal processing and statistical estimation \cite{SY98,CP10proximal}, image recovery \cite{Com96}, absolute value equations \cite{ACT23}, matrix completion \cite{RSV12} and robust PCA \cite{MA18}, SVMs \cite{LJS+13} and clustering \cite{BRZ23}.

A natural question is whether the good conditioning of individual polytopes can be lifted to the product domain in a way that still allows for linear convergence guarantees. We provide a surprising positive answer as the main contribution of this work: \emph{the product of well-conditioned polytopes inherits a form of good conditioning}, that enables linear convergence of \FW{} algorithms for smooth, convex, \PL{}-objectives by using either the pyramidal width or the vertex-facet distance. Positive polytope condition numbers, such as the pyramidal width or the vertex-facet distance, typically capture how well the directions found by the \LMO{} in the polytope align with the idealized direction $x^\ast - x_t$ (from an iterate $x_t$ towards a minimizer $x^\ast$) relative to the direction of the gradient that was used for the \LMO{}, cf. \cref{a:linconv_pw}. Our findings imply that this alignment does not degrade much when considering a product polytope, with respect to the individual components.

We apply our results to the aforementioned important primitive in large-scale optimization: approximately finding a feasible point in the intersection of several polytopes or certifying that such intersection is empty. In other words, the problem of interest is, given $k \geq 2$ polytopes $\mP_i$,
\begin{equation}
\text{find } x \in \bigcap_{i=1}^k \mathcal{P}_i\,, \tag{P}
\label{prob:cfp}
\end{equation}
if it exists, or determine that the intersection is empty. An  \emph{$\epsilon$-approximate} solution of \eqref{prob:cfp} consists of finding a point whose Euclidean distance to the intersection is less than $\epsilon$, or determining that no such point exists. Finding an approximate feasible point in the intersection of $k$ polytopes can be turned into an approximate solver for linear programming over such intersection as feasible set, by performing a binary search objective value, see, e.g., \cite{DVZ22} and references therein.

For the above task, we consider the Away Frank-Wolfe (\newtarget{def:acronym_away_frank_wolfe}{\AFW{}}) algorithm \cite{GM86} (see \cref{algo:AFW} in \cref{a:algos}) over the Cartesian product, where we have a variable for each polytope and minimize an objective that penalizes each of the pairwise distances between the $k$ variables, cf. \eqref{prob:intersection_k}. Our objective is convex, smooth and satisfies the \PL{} inequality. Naturally, this approach uses the \LMO{} of the resulting product polytope that is the product of the corresponding  \LMO{} outputs for each polytope component, that can be computed in parallel.

Our technical contributions are the following.

\paragraph{Contributions}
\begin{itemize}
    \item Regarding polytope condition numbers, we show that, if the \textit{pyramidal widths} of two polytopes $\mathcal{P}_1$ and $\mathcal{P}_2$ are $\delta_{\mathcal{P}_1}$ and $\delta_{\mathcal{P}_2}$, that of the product is exactly $\delta_{\mathcal{P}_1} \delta_{\mathcal{P}_2} /\sqrt{\delta_{\mathcal{P}_1}^2 + \delta_{\mathcal{P}_2}^2}$. This yields a pyramidal width $\delta_\mathcal{P} = \bigomega{\frac{1}{\sqrt{k}}\min_{i\in[k]} \delta_{\mathcal{P}_i}}$ for the product $\mathcal{P} = \prod_{i\in[k]} \mathcal{P}_i$ of $k$ polytopes. Our proof uses the \textit{affine pyramidal width}, which we introduce as a new equivalent characterization of the pyramidal width that is easier to work with in our proofs. Similarly, for the \textit{vertex-facet distance}, we show $\vf{\mP} = \min_{i\in[k]} \vf{\mP_{i}}$.

   \item As a consequence of studying these two polytope condition numbers, we can quantify linear rates of convergence for the \AFW{} algorithm optimizing convex, smooth functions satisfying the \PL{} inequality over a product polytope, which inherits its conditioning from the components.
   \item We apply our results to the problem of approximately finding a feasible point in the intersection of several polytopes, by minimizing \eqref{prob:intersection_k}, a convex, smooth and \PL{} objective.
   \item We substantiate our theoretical findings with empirical evidence, demonstrating the algorithm's applicability across large-scale optimization tasks.
\end{itemize}

The novel characterization of the pyramidal width as defined by the affine pyramidal width is of independent interest, and we believe that it can offer new insights into the optimization over polyhedral sets.

\section{Other related work} \label{s:literature}

\paragraph{Polytope condition numbers}
A line of research has established nontrivial lower bounds for polytope condition numbers in key domains \cite{wirth2023approximate,CFK23}. Recent works have further refined affine-invariant convergence analyses by exploiting facial distance to obtain improved Frank–Wolfe rates \cite{wirth2025accelerated,pena2023affine}.

\paragraph{Convex feasibility via projections.}
For the convex feasibility problem, classical methods make use of projections as opposed to linear minimization oracles, including von Neumann’s alternating projections and its extensions \cite{Von49,Gin18}, and solve the intersection problem by repeatedly projecting onto each set. In general these approaches achieve only sublinear convergence \cite{BS16,BC11,BPW23}, though linear rates can be recovered under additional regularity or transversality assumptions on the intersection \cite{GPR67,BB93,BB96,BT03,DIL15,BLM21,BBS21}.

\paragraph{Frank–Wolfe for feasibility and block‐coordinate variants.}
Several works applied \FW{}‐type methods to the convex feasibility problem. Early distance‐minimization schemes by \citet{Wil68,Wol76} use FW methods without providing convergence rates. More recent block‐coordinate FW algorithms \cite{BPS15,BRZ23} exploit product‐structure by cycling or randomly selecting blocks, but without global linear rates unless combined with costly subroutines. An augmented‐Lagrangian FW approach under a PL condition was proposed in \cite{GPL18} for two blocks, though it requires stringent feasibility assumptions and does not easily extend to larger $k$.

\paragraph{FW convergence and fast convergence rates.}
At the core of Frank–Wolfe research is the study of fast \cite{garber2015faster,KDP20,wirth2023acceleration,wirth2025accelerated} and affine‐invariant linear rates \cite{LJ15,TTP22,pena2023affine,wirth2024fast}. To derive linear rates, three key premises \cite{PR18} appear throughout the literature.
The first premise involves a positive \emph{condition number of the polytope}.
The condition number is used to establish an inequality comparing (A) the alignment of the gradient with the direction taken by the algorithm, obtained by the \LMO{} with the gradient as direction, and (B) the alignment of the gradient with the idealized direction $x_t - x^\ast$, cf. \cref{thm:scaling_ineq}.

The second premise is an \emph{objective-function growth} condition from the minimizers.
The standard assumption in this case is strong convexity \cite{GM86,LJ15,GH15,TTP22}, although weaker conditions have been successfully applied (see, e.g., \cite{BS15,LJ15,wirth2024fast}).
The third premise is a condition that guarantees we can obtain descent on the objective, such as smoothness. Our work also makes sure of these three premises in order to achieve linear convergence rates.
\section{Notation and Preliminaries}\label{s:preliminaries}

We denote the set of vertices of an $n$-dimensional (bounded) polytope $\mP$ by $\verts{\mP}$, the convex hull and affine hull of its vertices by $\conv{\mP}$ and $\aff{\mP}$, respectively, the set of its proper faces by $\faces{\mP}$ and the set of its facets by $\facets{\mP}$.
Recall that a face
is \emph{proper} if it is neither $\mP$ itself nor the empty set.
Facets are proper ``maximal'' faces, in the sense that they are not contained in any other face, so they have dimension $n-1$.

Throughout, we let $\|\cdot\|$ be the Euclidean norm and
$\dist{\mP, \mQ} \defi \min_{x \in \mP, y \in \mQ} \|x - y\|$ for nonempty polytopes $\mP, \mQ \subseteq \R^n$.
Further, we use the shorthand $\diam{\mP} \defi \max_{v, w \in \mathcal{P}} \|v - w\|$ to denote the \emph{diameter} of a polytope $\mathcal{P}$.
We also denote the Kronecker product of two matrices $A, B$ by $A \otimes B$.

We address the problem $\min_{x \in \mP} f(x)$ of minimizing a convex, smooth and \PL{} function over a polytope by considering \FW{} methods. Let $\Omega^\ast \defi \argmin_{x \in \mX} f(x)$ be its solution set, and we denote an arbitrary point in $\Omega^\ast$ by $x^\ast$.

Further, let $[k] \defi \{1, 2, \dots, k\}$. We denote the 
$m \times m$ identity matrix by $I_m$ and the $m$-dimensional vector of ones by $\ones_m$.
We make use of the Cartesian product of $k$ $n$-dimensional polytopes $\Pi_{i \in [k]} \mP_i \defi \mP_1 \times \mP_2 \times \cdots \times \mP_k \subset \R^{n \times k}$ denoting its components by upper indices $x = (x^1, x^2, \dots, x^k)$, where each $x^i \in \mP_i$ is said to be a \emph{block variable}.
In the following, we will consider blocks $I_t = \{i\}$, for some $i \in [k]$, and the partial gradients $\nabla^i f(x) \defi (\nabla f(x))^i$ of a differentiable function $f$ over block $i \in [k]$.
Recall that a face is a product of faces in this setting.

We will work with the following function regularity conditions, where the differentiability assumption over non-open sets $\mX$ is to be understood as differentiability in an open set containing $\mX$.


\begin{definition}[\textbf{Convexity and $L$-smoothness}]\label{def:conv_smooth}
    Let $\mX \subseteq \R^n$ be a convex set and $f : \R^n \rightarrow \R$ differentiable in $\mX$. The function $f$ is convex (resp. $L$-smooth) in $\mX$ if $\circled{1}$ holds for all $x, y \in \mX$ (resp. $\circled{2}$):
\[
    0 \circled{1}[\leq] f(x) -  f(y) - \innp{\nabla f(y), x - y} \circled{2}[\leq] \frac{L}{2}\left\| x - y\right\|^2.
    \vspace{0.3cm}
\]
\end{definition}
Under convexity, $L$-smoothness is equivalent to $\|\nabla f(x) - \nabla f(y)\| \le L\|x - y\|$, cf. \cite{Bec17}.

\begin{definition}[\textbf{Polyak-\L{}ojasiewicz (PL) condition}]
\label{def:pl}
Let $\mX \subseteq \R^n$ be a set over which we minimize a function $f : \R^n \rightarrow \R$, differentiable in $\mX$.
Assume the solution set $\Omega^\ast$ exists and is nonempty, and let the minimum function value be $f^\ast \defi f(x^\ast)$, for $x^\ast \in \Omega^\ast$.
We say that $f$ is $\mu$-\PL{} in $\mX$ if, for all $x \in \mX$,
\[
\frac{1}{2} \|\nabla f(x)\|^2 \geq \mu (f(x) - f^\ast).
\]
\end{definition}

A well-known fact is that under convexity, $\mu$-\PL{} is equivalent to the so-called $\mu$-quadratic growth condition (\newtarget{def:acronym_quadratic_growth}{\QG{}}), which is $f(x) - f^\ast \geq \frac{\mu}{2}\min_{x^\ast\in\Omega^\ast }\norm{x-x^\ast}$ for every $x\in \R^n$, cf. \citep{KNS16}.
\section{Conditioning of Product Polytopes and Convergence}\label{s:premises}
In this section, we develop our theory providing bounds for how our polytope condition numbers of interest behave when considering product polytopes. At the end of this section we quantify the linear rates of convergence that ensue from our results.

We start by studying the pyramidal width, and we show that if each subpolytope $\mP_i$ is well-conditioned in terms of pyramidal width lower bounds, then the product polytope $\mP$ is also well-conditioned.
We note that several equivalent definitions of the pyramidal width can be found in the literature. We make use of the one given by \citet{PR18},
where it is termed \emph{facial distance}, while the original notions can be found in \citep{lacoste2013affine}. See also \cite{BS15, LJ15}.
\eli{ if we rely on their characterization, i think we should call the objects we work with facial distances and not pyramidal width. we literally define things via the face. i strongly feel that we should make a final pass and change things to facial distance, it is the more descriptive term.

SOMEONE ELSE:
This change can be good and it's very easy to do. However, which term is more widely used? The CG book uses the pyramidal width term.

EW: Let's leave it then for now.
}


For a polytope $\mathcal{P}$ and one of its faces $f$, we write $P^f\defi\conv{\verts{\mP}\setminus \verts{f}}$.

\begin{definition}[\textbf{Pyramidal width \cite{PR18}}]\label{def:pyrwidth}
Let $\mP$ be a polytope, with at least two vertices and let $f$ be a face of $\mP$. We define the \emph{pyramidal width in $f$} $\wface{f}{\mP}$ of $\mP$ as the distance between $f$ and $\mP^f$. The \emph{pyramidal width} $\w{\mP}$ of $\mP$ is the minimum of $\wface{f}{\mP}$ over all faces $f$ of $\mP$. That is,
\begin{equation}\label{eq:pyrwidth_X}
\pw{\mP} \defi \min_{f \in \faces{\mP}} \dist{f, \conv{\verts{\mP} \setminus f}}.
    \vspace{0.3cm}
\end{equation}
\end{definition}

In order to achieve linear rates with FW algorithms it is enough to obtain lower bounds for the pyramidal width $\delta_{\mP}$ of the domain $\mP$. For a product polytope, $\mP = \prod_{i\in[k]}\mP_i$, we trivially have that $\delta_{\mP} \leq \min_{i\in[k]}\{\delta_{\mP_i}\}$ from the definition. In the sequel, we will show
a lower bound that is only a minor polynomial factor away from this upper bound, depending on $k$.
To that end, we introduce the following variant of the pyramidal width, which we show it is equivalent to the definition above.
\begin{definition}[\textbf{Affine pyramidal width}]
  Let $\mP$ be a polytope, and let $f$ be a face of $\mP$. The \emph{affine pyramidal width} of $\mP$ in $f$ is the distance between $\aff{f}$ and $\mP^f$. We denote it by $\whatface{f}{\mP}$.
  The \emph{affine pyramidal width} $\what_{\mP}$ of $\mP$ is the minimum of $\whatface{f}{\mP}$ over all faces $f$ of $\mP$. That is,
\begin{equation}\label{eq:affine_pyrwidth_X}
    \what_{\mP} \defi \min_{f \in \faces{\mP}} \dist{\aff{f}, \conv{\verts{\mP} \setminus f}}.
    \vspace{0.3cm}
\end{equation}
\end{definition}

Observe that, in these two definitions, one never considers the trivial faces $\emptyset$ and $\mP$. Allowing either would yield an infinite pyramidal width under the convention that the distance between a set and the empty set is infinity. Since $\mP$ is at least $1$-dimensional, a nontrivial face exists so the pyramidal width is finite.

We first show that the pyramidal width is always realized in the relative interior of a face, which implies that the two definitions above yield the same quantity. However, the affine pyramidal width that we introduced will be more convenient for our compositional proofs in product polytopes.

We note
that $\whatface{f}{\mP}$ is not always equal to $\wface{f}{\mP}$. Our claim only applies to $\w{\mP} = \what_{\mP}$, that is, to the minimum over all faces for each respective quantity.

\begin{lemma}\linktoproof{lemma:affine_pyramidal_width} \label{lemma:affine_pyramidal_width}
  Let $\mP$ be a full-dimensional polytope in $\R^n$ for $n\geq 1$, and let $f$ be a face of $\mP$ that achieves the minimum affine pyramidal width. Let $p$ be a point of $\aff{f}$ that realizes the minimum distance.
  Then, $p$ is in the relative interior of $f$.
\end{lemma}

Indeed, having proved that the point realizing the minimum distance lies in the relative interior of the corresponding face readily implies that the two pyramidal width notions coincide.

\begin{corollary}[\textbf{Widths equivalence}]\label{cor:affinepyr_sameas_pyr}
  Let $\mP$ be a polytope. Then,
  \[
    \w{\mP} = \what_{\mP}.
    \vspace{0.3cm}
  \]
\end{corollary}
\begin{proof}
    For every face $f$ of $\mP$, we straightforwardly have $\whatface{f}{\mP}\leq \wface{f}{\mP}$ from the definition. Hence, $\what_{\mP}\leq \w{\mP}$. Now let $f$ be a face of $\mP$ with minimum affine pyramidal width. Then, by \cref{lemma:affine_pyramidal_width}, $\wface{f}{\mP}=\whatface{f}{\mP}$. Hence, $\w{\mP}=\what_{\mP}$.

\end{proof}

The pyramidal width $\pw{\mP,\mQ}$ of the Cartesian product $\mP \times \mQ$ of two polytopes is characterized as follows. In the proof, we find specific faces that realize this distance and prove its optimality, yielding an equality.

\begin{theorem}[\textbf{Pyramidal width of the product}]
\linktoproof{prop:pyrwidth_mPmQ}\label{prop:pyrwidth_mPmQ}
Let $\pw{\mP}$ and $\pw{\mQ}$ be the pyramidal widths of polytopes $\mP, \mQ \subseteq \mathbb{R}^n$. Then,
\begin{equation}\label{eq:pyrwidth_mPmQ}
\pw{\mP \times \mQ} = \sqrt{\frac{\pw[2]{\mP}\pw[2]{\mQ}}{\pw[2]{\mP} + \pw[2]{\mQ}}}\,.
    \vspace{0.3cm}
\end{equation}
\end{theorem}

Note that \cref{prop:pyrwidth_mPmQ} directly implies a bound $\delta_\mathcal{P} = \bigomega{\frac{1}{\sqrt{k}}\min_{i\in[k]} \delta_{\mathcal{P}_i}}$ for the pyramidal width of the product polytope $\mathcal{P} = \prod_{i\in[k]} \mathcal{P}_i$ if $\delta_{\mathcal{P}_i}$ is the pyramidal width of $\mP_i$, see \cref{cor:pyr_width_lb_for_k_polytopes}. Compare this lower bound with the trivial upper bound $\pw{\mP} \leq \min_{i\in[k]}\{\pw{\mP}_i\}$: the upper bound is essentially tight when one pyramidal width is much smaller than the others.





We also derive rates that rely on the vertex-facet distance, which is defined as in the following, where now the vertices of the polytope play a crucial role.

\begin{definition}[\textbf{Vertex-facet distance \cite{PR18}}]\label{def:vfd}
Let $\mP$ be a nonempty polytope with at least two vertices. We define the vertex-facet distance of $\mP$ as
    \vspace{0.1cm}
\[
\vf{\mP} \defi \min_{F_\mP \in \facets{\mP}} \dist{\aff{F_\mP}, \verts{\mP}\setminus F_\mP},
\]
    \vspace{0.1cm}
where $\verts{\mP}$ is the set of vertices of $\mP$.
\end{definition}

The fact that this polytope condition number restricts to distances between facets and vertices makes it behave better for product polytopes. In particular, we do not get a polynomial factor decrease with respect to the minimum condition number, as it was the case for the pyramidal width.

\begin{proposition}[\textbf{Vertex-facet distance of the product}]\linktoproof{prop:vfd_PxQ_bounds}\label{prop:vfd_PxQ_bounds}
Let $\mP \defi \mP_1 \times \cdots \times \mP_k$ be a product polytope. We have $\vf{\mP} = \min_{i \in[k]}\left\{\vf{\mP_i}\right\}$.
\end{proposition}

An interesting fact, proven in \cite[Theorem 2.22]{RS22},  is that for a polytope $\mP$ with at least two vertices, the following inequality holds: $\vf{\mP} \ge \pw{\mP}$.


%
%
%
%
%
%
%
%
%
%
\subsection{Linear Convergence}\label{s:convergence}

In this section, we examine the convergence behavior of the \AFW{} algorithm, when is run to minimize a \PL{} objective over a product polytope.
The following quantifies the linear rates of convergence that are obtained depending on the pyramidal widths or the vertex-facet distances of the individual polytopes in the product.

\begin{proposition}[\textbf{Linear rates via the product's pyramidal width}]
\label{prop:linconv_pw}\linktoproof{prop:linconv_pw}
For $i\in[k]$, let $\mP_i$ be a polytope with pyramidal width $\pw{\mP_i}$.
The \AFW{} algorithm obtains an $\epsilon$-minimizer of a convex, $L$-smooth and $\mu$-\PL{} objective over $\Pi_{i \in [k]}\mP_i$ in at most
    \vspace{0.2cm}
\[
t= \bigol{\frac{L\diam{\mP}^2}{\mu \alpha_{\pw{}}^2} \log(\frac{L\diam[2]{\mP}}{\epsilon})}
    \vspace{0.3cm}
\]
iterations, where $\diam[2]{\mP}$ is the diameter of $\Pi_{i \in [k]}\mP_i$, and $\alpha_{\pw{}} \defi \frac{1}{\sqrt{2}^{\ceil{\log_2 (k+1)}}}\min_{i \in [k]} \{\pw{\mP_i}\}$.
%
\end{proposition}


We quantify linear convergence rates by using the composition of vertex-facet distances in the product by making use of the analysis in \cite{BS15}. The authors did not state their results for \PL{} functions, but their proofs work in this case.

\begin{proposition}[\textbf{Linear rates via the product's vertex-facet distance}]
\label{prop:linconv_vf}\linktoproof{prop:linconv_vf}
For $i\in[k]$, let $\mP_i$ be a polytope with vertex-facet distance $\vf{\mP_i}$.
The \AFW{} algorithm obtains an $\epsilon$-minimizer of a convex, $L$-smooth,  $\mu$-\PL{} objective $f(x)$ over $\Pi_{i \in [k]}\mP_i$
    in at most
    \vspace{0.2cm}
\[
t= O\left(\min\left\{ \frac{32L\diam{\mP}^2N^2}{\mu \alpha_{\vf{}}^2}, 4 \right\}\log(\frac{f(x_0)-f(x^\ast)}{\epsilon})\right)
\]
    \vspace{0.2cm}
iterations
where $N$ is a constant associated with a procedure performed at the end of each iteration of the \AFW{} algorithm, described in \cref{a:algos}.
%
\end{proposition}


%
%
%
%
%
%
%
%
%
%
%
%
%
%
%
%
%
%
%
%
\section{The approximate feasibility problem for a polytope intersection}\label{s:polytope_intersection}

\Dav{Slightly mention some context here although the main related work should go in the related work section}

In this section, we present the objective that we can optimize in order to approximately solve the feasibility problem \eqref{prob:cfp} for polytopes. We have a variable for each polytope component and we penalize the variables being far from each other.
    \vspace{0.3cm}
\begin{equation}\label{prob:intersection_k}
    f(x) \defi \frac{1}{2k} \sum_{i=1}^{k-1} \sum_{j=i+1}^k \|x^i - x^j\|^2\, \text{ for } x \in \mP,
    \vspace{0.3cm}
\end{equation}
where $\mP \defi \Pi_{i \in [k]} \mP_i$, and each $\mP_i$ is a polytope of dimension $n$ so $x = (x^1, x^2, \ldots, x^k) \in \R^{nk}$.
Note that $f(x) = \frac{1}{2k} \innp{M_k x, x}$, with $M_k \defi (k I_{k\times k} - \ones_{k} \ones_{k}^T) \otimes I_n$.


Now we establish several properties of the objective function of \cref{prob:intersection_k} that show that we can apply the theory in \cref{s:premises} to determine linear convergence rates for our application.

\begin{proposition}[\textbf{Convexity, smoothness, and PL of \eqref{prob:intersection_k}}]\linktoproof{prop:lsmooth_cvx_pl_intersection_k}\label{prop:lsmooth_cvx_pl_intersection_k}
    For $i \in [k]$, let $\mP_i \subseteq \R^n$ be convex sets and $\mP \defi \Pi_{i \in [k]} \mP_i$.
    The function $f(x) \defi \frac{1}{2k} \sum_{i=1}^{k-1} \sum_{j = i+1}^k\|x^i -x^j\|^2$ in \eqref{prob:intersection_k} is convex, $1$-smooth, and  $1$-\PL{} over $\mP$.
\end{proposition}

We note that due to the $1$-PL property and convexity, our objective also satisfies the $1$-quadratic growth condition. Thus, if the intersection of polytopes is non-empty, the optimum value for the objective in \eqref{prob:intersection_k} is $0$, and reaching an $\frac{\epsilon}{2k}$ minimizer of $f$ implies that each of our $k$ variables is an $\epsilon$-approximate feasible point. Note that this factor appears in a logarithmic factor in the convergence rates.

On the other hand, if the intersection is empty, we can run the algorithm until we observe that the function value is larger than a primal-dual gap that is computed by the algorithm and is decreasing with linear rates, which guarantees that the optimal value is strictly positive, cf. \cref{a:algos}.

\dav{We should actually be more explicit about the primal-dual measure, and in fact for the vertex facet case we did not dothe primal-dual proof. The one of the pyramidal width is sketched in  \citep{martinez2025beyond}. It should be done after the deadline and it is also a nice addition to the theory of the paper.}
\section{Experiments}\label{s:experiments}

In this section, we evaluate several algorithms for solving the problem in \cref{prob:intersection_k} on a product polytope, showcasing their practical performance and the differences among the explored options.
Our experiments compare the \FW{} and \AFW{} algorithms, cf. \cref{a:algos}, with different heuristic variants.

We implement two types of algorithms: ``full'' (F-\FW{} and F-\AFW{}), that at each iteration computes a single \LMO{} over the entire product polytope $\mP$ and maintains a single \emph{active set} in the product, which is a set of previously computed extreme points stored such that the current iterate can be written as a convex combination of such points,
and ``cyclic block-coordinate'' (C-BC), where at each iteration we only compute the \LMO{} of one component and update the corresponding variable and active set only.


We observe that the full \AFW{} variant enjoys the best performance among the tested algorithms.
At each iteration $t$, F-\AFW{} optimizes over the active set, to find an away vertex, and updates the set, see \cref{algo:AFW}.
An away vertex $a_t$ allows the algorithm to choose to move in the direction guaranteeing more progress, namely, taking an \emph{away step} from $x_t$ in the direction $x_t - a_t$, or taking a \emph{\FW{} step} from $x_t$ towards the \LMO{} solution.
This not only accelerates the algorithm in practice in high-dimensional settings, but it is key for our linear convergence analysis, as \cref{fig:k2n20000,fig:k10n10000} show.
All algorithms were run with a standard line search, see \cref{a:algos}.

\begin{figure}[ht!]
  \centering
  \hspace*{-0.075\textwidth} 
  \begin{adjustbox}{center}
    \includegraphics[width=1.15\textwidth]{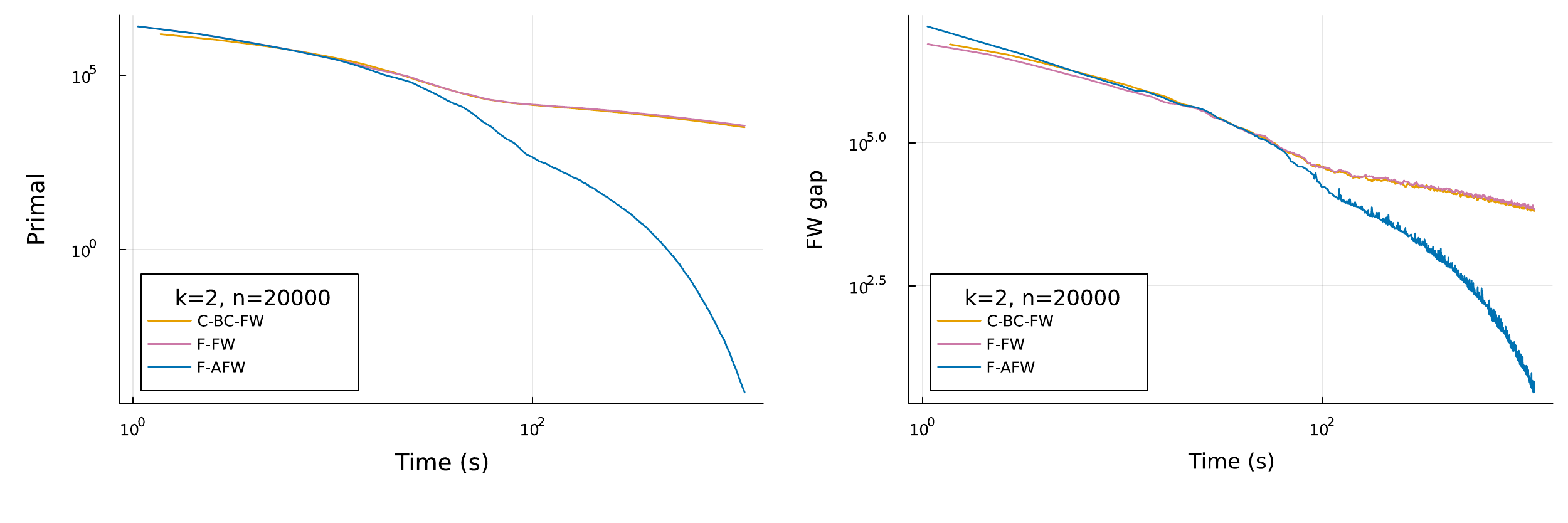}
  \end{adjustbox}
\caption{Different variants of the FW algorithm (C-BC-FW, F-FW, F-AFW) run for up to $1000$ iterations on two non‑intersecting polytopes in $\R^{20000}$. }
\label{fig:k2n20000}
\end{figure}

Here we report results for settings with $k = 2, 10$ polytopes in $\R^n$, with up to $n = 20000$. We present more plots with different settings in \cref{app:extra_experiments}.
Each instance is constructed by generating $k$ polytopes. We generate two kinds of problems: one where the polytopes do not intersect, and another where they do. In the first case, each polytope is generated  as the convex hull of a list of points that are sampled uniformly at random from mutually disjoint intervals for every coordinate $i \in [n]$. In the latter case, we shift the aforementioned polytopes $\mP_1, \mP_2, \dots, \mP_k$ so they
have non-empty intersection. Concretely, we shift $\mP_2$ until one of its vertices coincides with a vertex $v$ of $\mP_1$; we then move $\mP_2$ towards the barycentre of $\mP_1$ (the average of its vertices); finally, we translate all remaining polytopes so they share $v$.

The convergence happens to be faster across all algorithms for the non-intersecting setting. Compare \cref{fig:k2n10000_intersecting}, reporting algorithmic performances on the same polytopes used in \cref{fig:k2n20000} but made to intersect, and see the experiments in \cref{app:extra_experiments}, where we showcase more problems with non-empty intersection.

\begin{figure}[ht!]
  \centering
  \hspace*{-0.075\textwidth} 
  \begin{adjustbox}{center}
    \includegraphics[width=1.15\textwidth]{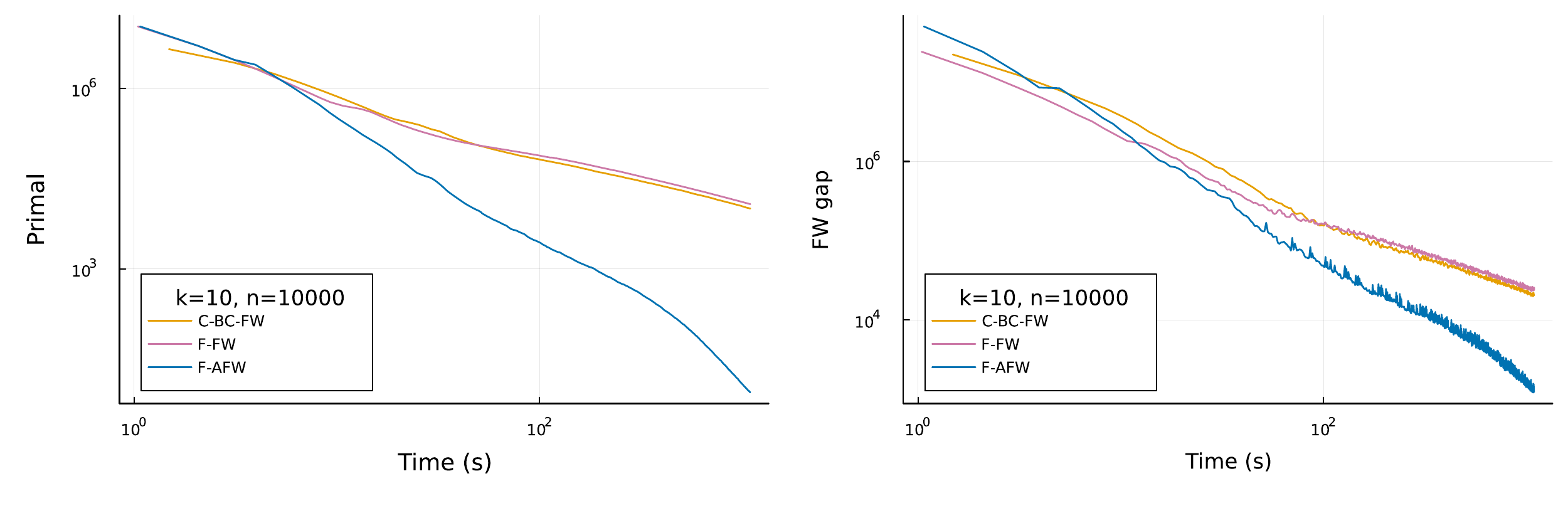}
  \end{adjustbox}
    \caption{Different variants of the FW algorithm (C-BC-FW, F-FW, F-AFW) run for up to $1000$ iterations on ten non‑intersecting polytopes in $\R^{10000}$. One can observe that the rates of convergence for the objective are at least linear in F-AFW.}
\label{fig:k10n10000}
\end{figure}

For each polytope, we draw the number of points uniformly at random in $[n, 2n]$. This choice keeps each polytope ``reasonably challenging'' to optimize over.
In \cref{app:extra_experiments}, we also report an example with two disjoint polytopes in $\R^{10,000}$ built from $100,000$ vertices, an order of magnitude larger: the convergence behavior of the tested algorithms in that instance is essentially unchanged, and only the absolute time to convergence increases, reflecting a higher cost for the \LMO{}.

In the plots, we report the following metrics to quantify algorithmic performance over iterations: the primal gap $f(x_t)-\min_{x\in\mP}f(x)$, and also the so-called \FW{} gap $\max_{x \in \mathcal{P}} \innp{\nabla f(x_t), x_t - x}$ \citep{BCC+23}. While the former is generally not computable, since one does not usually know the minimum objective value a priori, the latter is a standard measure for \FW{} algorithms, since it is an upper bound on the former while at the same time it is computable. Therefore, the \FW{} gap results in a stopping criterion that can be used to guarantee $\epsilon$-optimality has been reached. We used the FW gap as a stopping criterion in our experiments, if they already converged, cf. \cite{Jag13}. We ran every \FW{} variant up to a maximum of $1000$ iterations or until the \FW{} gap dropped below $1e-07$, whichever occurred first.
Both metrics in the figures below are plotted against
the elapsed computational time, in seconds.

The implementation was carried out fully in Julia (Version 1.9.4) and, in particular, we relied on the \href{https://github.com/ZIB-IOL/FrankWolfe.jl}{\texttt{FrankWolfe.jl}} package (Version 0.5.0, MIT License) for the algorithmic part.
All runs were executed on Slurm compute nodes equipped with Intel Xeon Gold 6342 processors, each providing 96 logical cores, and approximately 512GiB RAM, interconnected via Infiniband.
We requested $k$ cores per task in Slurm via \texttt{\#SBATCH\ --cpus-per-task=k}, and then set \texttt{export JULIA\_NUM\_THREADS=\$SLURM\_CPUS\_PER\_TASK} so that Julia’s \FW{} routines used exactly $k$ threads.
The code for reproducing the results in this section will be made publicly available upon publication and can be found in the supplementary material.

%


All in all, we found a consistent hierarchy across all experimental settings.
Even as the number of polytopes grows, \AFW{} maintains a clear edge over the other variants, although this advantage gradually diminishes, particularly in terms of the measured \FW{} gap.
Instead, the algorithmic performance of the tested \FW{}-based variants F-\FW{} and C-BC-\FW{} is very similar, and it yields worse convergence than F-\AFW{}.
Surprisingly, on our instances, the cyclic block‑coordinate \FW{} strategy offers no practical gain over the full \FW{} variant, even with large $k$.
For the case of intersecting polytopes in \cref{fig:k2n10000_intersecting} and others in \cref{app:extra_experiments}, the problem becomes easier and the difference among algorithms becomes narrower.

\begin{figure}[ht!]
  \centering
  \hspace*{-0.075\textwidth} 
  \begin{adjustbox}{center}
    \includegraphics[width=1.15\textwidth]{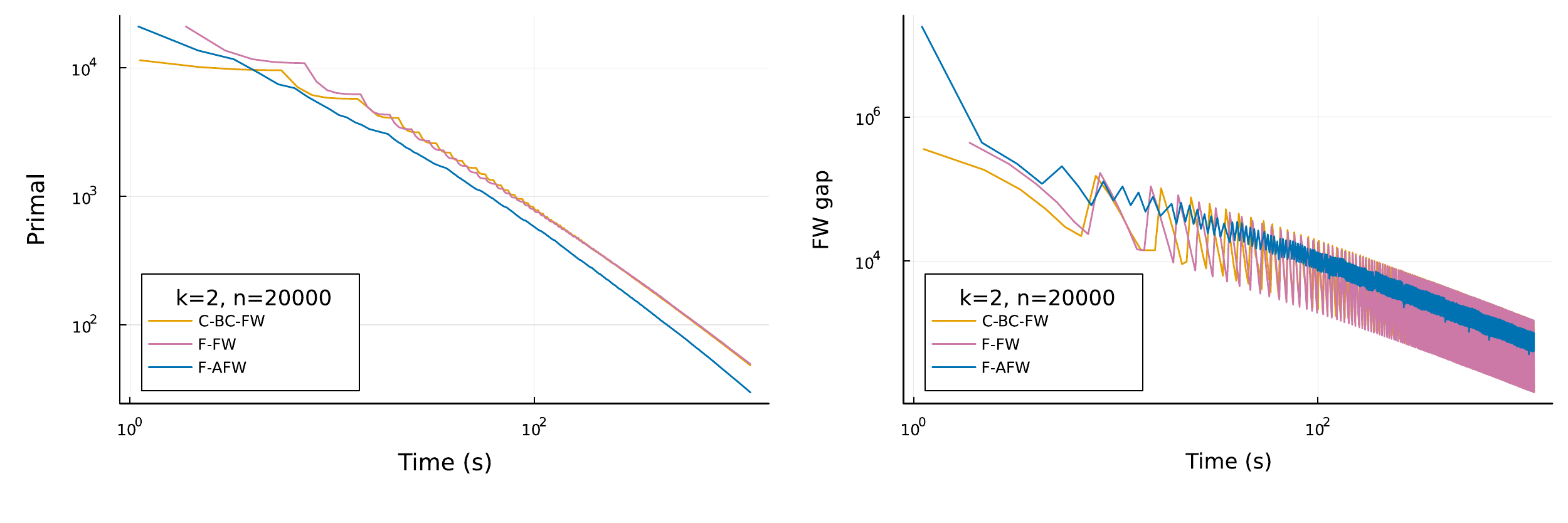}
  \end{adjustbox}
\caption{Different variants of the FW algorithm (C-BC-FW, F-FW, F-AFW) run for up to $1000$ iterations on two intersecting polytopes in $\R^{20000}$. Note that the FW gap is not guaranteed to decrease monotonically but the minimum of it across iterations decreases linearly as well.}
\label{fig:k2n10000_intersecting}
\end{figure}


We note that we also experimented with a stochastic block-coordinate \FW{} variant, which differs from C-BC-\FW{} by selecting blocks in random order instead of cyclically. In our tests, its primal and FW gap trajectories were almost exactly the same as those of C-BC-\FW{}, so we decided to omit it from our figures.

\cref{fig:k2n20000,fig:k10n10000} showcase the superiority of
employing \AFW{} variants in speeding up convergence.
Most importantly, \emph{our computational results provide empirical confirmation of our theoretical linear rates for \AFW{}} in several scenarios, as this algorithm converged consistently faster than F-\FW{}
and C-BC-\FW{} in all settings on disjoint polytopes.
\section{Conclusion}

In this work, we have shown that the product of well‐conditioned polytopes inherits a robust notion of conditioning, both in terms of pyramidal width and vertex–facet distance, which in turn guarantees linear convergence of the Away Frank-Wolfe method under standard assumptions. Our analysis yields explicit bounds on the resulting condition numbers of the product domain, and thus on the convergence rate, as a function of the individual polytope components. We then applied this theory to the classical convex-feasibility problem over intersections of polytopes by minimizing a smooth, convex, Polyak-\L{}ojasiewicz objective over the Cartesian product. We provided empirical results that strongly corroborate our findings.
\begin{ack}
The research in this paper was partially supported through the Research Campus Modal funded by the German Federal Ministry of Education and Research (fund numbers 05M14ZAM,05M20ZBM) and the Deutsche Forschungsgemeinschaft (DFG) through the DFG Cluster of Excellence MATH+.
David Martínez-Rubio was partially funded by the project IDEA-CM (TEC-2024/COM-89).
Francisco Criado was supported by grant PID2022-137283NB-C21.
\end{ack}
\printbibliography[heading=bibintoc] %
\clearpage
\appendix
\section{Proofs}\label{a:proofs}
\subsection{Pyramidal width}\label{a:pw}

\cref{fact:cart_prod_kpolytopes} contains some simple known facts about polytopes, that serve as a warm-up and will be useful in the sequel.

\begin{lemma}[\textbf{Cartesian product of $k$ polytopes}]\label{fact:cart_prod_kpolytopes}
The Cartesian product $\mP = \Pi_{i \in [k]} \mP_i$ of $k$ polytopes, each in $\R^n$ for some $n \in \mathbb{N}_{>0}$, is a polytope of dimension $kn$.
In particular, $\verts{\mP} = \Pi_{i \in [k]}\verts{\mP_i}$.
Further, if $F_\mP \in \faces{\mP}$, then  $F_\mP$ is the product of faces of all the individual polytopes, i.e., $\forall i \in [k]$ there exist $F_{\mP_i} \in \faces{\mP_i}$ such that $F_\mP = \Pi_{i \in [k]} F_{\mP_i}$.
\end{lemma}

\begin{proof}
For all $i \in [k]$, we denote the $i$-th polytope by $\mP_i \defi \{x \in \R^n : A_{\mP_i} x \le b_{\mP_i}\}$, where $A_{\mP_i} \in \R^{m_{\mP_i} \times n}$ and $b_{\mP_i} \in \R^{m_{\mP_i}}$.
Compactness and convexity are inherently preserved in the Cartesian product, so
\[
\begin{array}{rcl}
\Pi_{i \in [k]} &=& \{(x^1, \dots, x^k) \in \R^n\times\cdots \times\R^n\,:\, x^1 \in \mP_1, \dots, x^k \in \mP_k\}\\
&=& \{(x \in \R^{kn}\,:\, \forall i \in [k]\,,\, A_{\mP_i} x^i \le b_{\mP_i}\}\,.
\end{array}
\]
By linearity of the constraints and independence of each block $i \in [k]$,
the product is still a polytope described by a constraint matrix in $\R^{(m_{\mP_1}+\dots+m_{\mP_k})\times kn}$ with diagonal blocks $A_{\mP_i}$, for $i \in [k]$, and right-hand side coefficients $b_{\mP_i} \in \R^{m_{\mP_1}+\dots+m_{\mP_k}}$.
Moreover, since $\forall i \in [k]$, $\mP_i = \conv{\mP_i}$ and $\conv{\Pi_{i \in [k]}\mP_i} = \Pi_{i \in [k]}\conv{\mP_i}$, then $\verts{\mP} = \Pi_{i \in [k]}\verts{\mP_i}$.
Finally, a polytope face is the set of points optimizing some linear functional over the polytope.
Consider the linear functional defined by $n = (n^1, \ldots, n^k) \in \R^k$.
A point $(p^1, \ldots, p^k) \in \mP$ optimizes the linear functional defined by $n$ if and only if, each $p^i$ optimizes the linear functional defined by the corresponding $n^i$ over $\mP_i$, as a linear program in the product is solved in each component independently.
If, for all $i \in [k]$, we denote the faces defined by that condition as $F_i \subseteq \mP_i$, the set of points optimizing $n$,is exactly $F_1 \times F_2 \times \cdots \times F_k \subset \mP_1 \times \mP_2 \times \cdots \times \mP_k$, i.e., a face $F_\mP$ of $\mP$.
\end{proof}

\begin{proof}\linkofproof{lemma:affine_pyramidal_width}

Suppose that $p$ is not in the relative interior of $f$. Then $p$ might be in the exterior or in the boundary of $f$.
Choose one facet $g$ of $f$ such that $\aff{g}$ separates
$p$ from the interior of $f$ ($p$ could lie exactly in $\aff{g}$ if $p$ is in the relative boundary of $f$). We will show that $\whatface{g}{\mP}<\whatface{f}{\mP}$, which would be a contradiction.

\begin{figure}[ht!]
\centering
\includegraphics[width=0.8\textwidth]{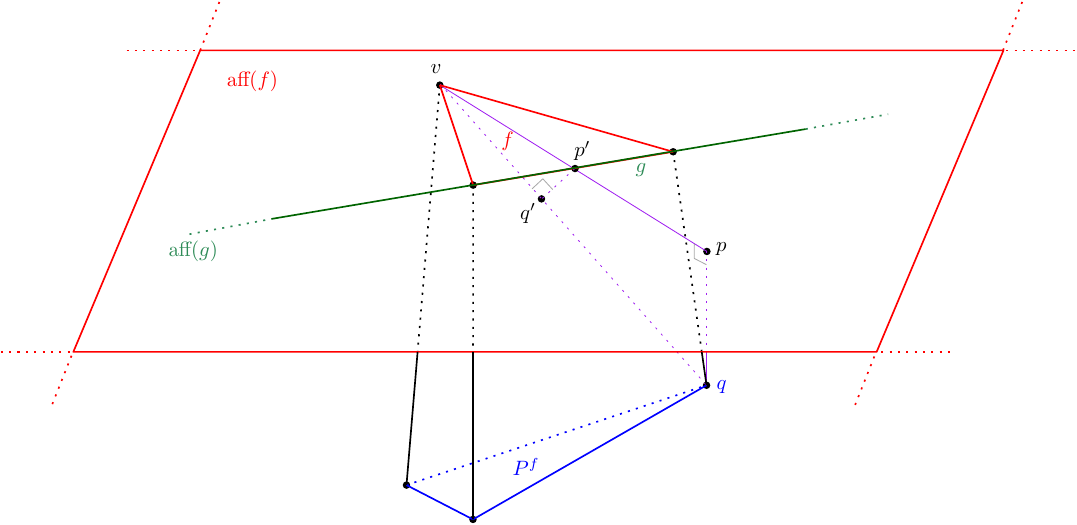}
\caption{\gabrev{To prove that $\dist{\aff{g}, [q,v]} < \dist{\aff{f},q}$, we looking at the right triangle $v,p,q$ and compare $\dist{p',q'}$ with $\dist{p,q}$.}
}
\label{fig:relint_pw}
\end{figure}

Since $g$ is a face of $f$, there is some vertex $v$ of $f$ that is not a vertex of $g$, which means $v\not\in \aff{g}$. Let $p'$ be the intersection of the segment $[p,v]$ with the affine subspace $\aff{g}$. This intersection exists and is unique because $\aff{g}$ separates $p$ and $v$, $[p,v]$ is $1$-dimensional, and $\aff{g}$ is $1$-codimensional as a subspace of $\aff{f}$.

Let $q$ be the point of $\mP^f$ that is closest to $p$. Finally, let $q'$ be the point of the segment $[v,q]$ that is closest to $\aff{g}$. By construction, $q'\in \mP^g$. \gabrev{Observe from \cref{fig:relint_pw} that, in the right triangle $v,p,q$, the segment corresponding to $\dist{p',q'}$ is at most as long as its height $\dist{p,q}$ over the hypotenuse, giving the key inequality:
}
%
\[ \whatface{g}{\mP} \leq \dist{\aff{g},q'} \leq \dist{p',q'} < \dist{p,q} = \whatface{f}{\mP}. \]
\end{proof}

In the sequel we will use $\w{\mP}$ to denote both the affine pyramidal width and the pyramidal width. However, we will use the definition of affine pyramidal width to prove the following theorem.


\begin{proof}\linkofproof{prop:pyrwidth_mPmQ}
  We will first prove the upper bound by explicitly constructing the face of the product that achieves the affine pyramidal width we claim, as well as the points that achieve the minimum distance. In the second part of the proof, we will prove that no other face of the product can have a smaller affine pyramidal width.

  Let $f_1$ be a face of $\mP_1$ such that $\w{\mP_1}=\whatface{f_1}{\mP_1}$. Let $p_1\in f_1$, $q_1\in \mP_1^{f_1}$, and $\delta_1\in \R$ be such that $\dist{p_1,q_1}=\delta_1=\w{\mP_1}$. Note that we know that $p_1\in f_1$ and not just $p \in \aff{f_1}$ because of \cref{lemma:affine_pyramidal_width} and the optimality of $f_1$. We define $f_2$, $p_2$, $q_2$, and $\delta_2$ analogously.

  Let
  \begin{align*}
    p&=(p_1,p_2), \\
    q&=(p_1,q_2)\frac{\delta_1^2}{\delta_1^2+\delta_2^2} + (q_1,p_2)\frac{\delta_2^2}{\delta_1^2+\delta_2^2}. \end{align*}

  We have that $p\in f_1\times f_2$ by construction. Now, we also have $q\in \left(\mP_1\times \mP_2\right)^{f_1\times f_2}$. Indeed, $q$ is a convex combination of the points $(p_1,q_2)$ and $(q_1,p_2)$. The point $(p_1,q_2)$ is in $f_1 \times \mP_2^{f_2}$, which is a subset of $\left(\mP_1\times \mP_2\right)^{f_1\times f_2}$. Analogously, $(q_1,p_2)$ is also in $\left(\mP_1\times \mP_2\right)^{f_1\times f_2}$. Thus, their convex combination is also in $\left(\mP_1\times \mP_2\right)^{f_1\times f_2}$.

  Now, let us compute $\dist{p,q}$, which is an upper bound for $\w{\mP_1\times \mP_2}$:
  \begin{align*}
    \dist{p,q} &= \left\| (p_1, p_2) - \left((p_1,q_2)\frac{\delta_1^2}{\delta_1^2+\delta_2^2} + (q_1,p_2)\frac{\delta_2^2}{\delta_1^2+\delta_2^2}\right) \right\| \\
                  &= \left\| \left(\left(p_1 - q_1\right)\frac{\delta_2^2}{\delta_1^2+\delta_2^2}, \left(p_2 - q_2\right)\frac{\delta_1^2}{\delta_1^2+\delta_2^2} \right) \right\| \\
                  &= \sqrt{\frac{\left(\delta_2^2+\delta_1^2\right)\delta_1^2\delta_2^2}{\left(\delta_1^2+\delta_2^2\right)^2}}\\
                  &= \frac{1}{\sqrt{\frac{1}{\delta_1^2}+\frac{1}{\delta_2^2}}}.
  \end{align*}

  Now let us prove the lower bound. This time, let $f_1$ and $f_2$ be \emph{arbitrary} faces of $\mP_1$ and $\mP_2$ and denote $\what_1 \defi \whatface{f_1}{\mP_1}$ and $\what_2 \defi \whatface{f_2}{\mP_2}$.

  Since $\dist{\aff{f_1},\mP_1^{f_1}} = \what_1$ and $\dist{\aff{f_2},\mP_2^{f_2}} = \what_2$, by the Hyperplane Separation Theorem there are linear functionals:

  \begin{align*}
      \begin{aligned}
    \pi_1:&\R^n \to \R, &\pi_1(x) = \langle u_1, x \rangle + c_1,      \\
    \pi_2:&\R^m \to \R, &\pi_2(y) = \langle u_2, y \rangle + c_2,
      \end{aligned}
  \end{align*}
  such that $\pi_1(x) = \what_1$, $\pi_2(y) = \what_2$ for all $x\in \aff{f_1}$ and $y\in \aff{f_2}$, and $\pi_1(x) \leq 0$, $\pi_2(y) \leq 0$ for all $x\in \mP_1^{f_1}$, $y\in \mP_2^{f_2}$. Furthermore $|u_1|=|u_2|=1$. Observe that $\pi_1$, $\pi_2$ are constant over $\aff{f_1}$ and $\aff{f_2}$ respectively. Indeed, for example if $\pi_1$ were not constant over $\aff{f_1}$ then $\pi_1(\aff{f_1})=\R$ so $\pi_1$ would not be a separating functional.

  We define the linear functional:
  \begin{align*}
    \pi: \R^{n+m} \to \R, \quad \pi(x,y) = \frac{\frac{\pi_1(x)}{\what_1}+\frac{\pi_2(y)}{\what_2}-1}{\sqrt{\frac{1}{\what_1^2}+\frac{1}{\what_2^2}}}= \left\langle \frac{\left(\frac{u_1}{\what_1}, \frac{u_2}{\what_2}\right)}{\sqrt{\frac{1}{\what_1^2}+\frac{1}{\what_2^2}}}, (x,y) \right\rangle + c.
  \end{align*}

  Observe that the gradient of $\pi$ is unitary by construction. Also, $\pi(x,y)=\frac{1}{\sqrt{\frac{1}{\what_1^2}+\frac{1}{\what_2^2}}}$ for all $(x,y)\in \aff{f_1}\times \aff{f_2}$. Now let us study $\pi(x,y)$ over the polytope $\left(\mP_1\times \mP_2\right)^{f_1\times f_2}$.

  Since $\left(\mP_1\times \mP_2\right)^{f_1\times f_2}$ is a polytope, $\pi$ is maximized at a vertex of $\left(\mP_1\times \mP_2\right)^{f_1\times f_2}$. The vertices of this polytope are of the form $(v_1,v_2)$ where $v_1\in \verts{\mP_1}$ and $v_2\in \verts{\mP_2}$ but removing the vertices of $f_1\times f_2$. In other words, $v_1\not\in \aff{f_1}$ or $v_2\not\in \aff{f_2}$.

  Without loss of generality we consider the case where $v_2\in \verts{\mP_2}$. Then, $\pi_1(v_1)\leq \what_1$ and $\pi_2(v_2)\leq 0$, which implies $\pi(v_1,v_2)\leq 0$. Therefore, $\pi$ is a separating functional for $\left(\mP_1\times \mP_2\right)^{f_1\times f_2}$ and $\aff{f_1}\times \aff{f_2}$.

  Recall that $\pi$ is unitary, so the distance between $\left(\mP_1\times \mP_2\right)^{f_1\times f_2}$ and $\aff{f_1}\times \aff{f_2}$ is at least $\frac{1}{\sqrt{\frac{1}{\what_1^2}+\frac{1}{\what_2^2}}}$. In other words, $\whatface{f_1\times f_2}{\mP_1\times \mP_2}\geq \frac{1}{\sqrt{\frac{1}{\what_1^2}+\frac{1}{\what_2^2}}}$.

  By repeating this argument for every pair of faces of $\mP_1$ and $\mP_2$:

   \begin{align*}
     \w{\mP_1\times \mP_2} &= \min_{\substack{f_1\in F(\mP_1)\\ f_2\in F(\mP_2)}} \whatface{f_1\times f_2}{\mP_1\times \mP_2} \\
               &\geq \min_{\substack{f_1\in F(\mP_1)\\ f_2\in F(\mP_2)}} \left(\frac{1}{\sqrt{\frac{1}{\left(\whatface{f_1}{\mP_1}\right)^2} + \frac{1}{\left(\whatface{f_2}{\mP_2}\right)^2} }} \right) \\
               &= \frac{1}{\sqrt{\frac{1}{\left(\min_{f_1\in F(\mP_1)}\whatface{f_1}{\mP_1}\right)^2}+\frac{1}{\left(\min_{f_2\in F(\mP_2)}\whatface{f_2}{\mP_2}\right)^2}}} \\
               &= \frac{1}{\sqrt{\frac{1}{\w{\mP_1}^2}+\frac{1}{\w{\mP_2}^2}}}.
   \end{align*}

   This concludes the proof.
\end{proof}

\begin{lemma}\label{lemma:preparatory_bounds_pw}
Let $a, b > 0$. Then
\begin{equation}
\frac{1}{\sqrt{2}}\min\{a, b\} \le \sqrt{\frac{a^2b^2}{a^2 + b^2}} < \min\{a, b\}
\label{eq:preparatory_bounds_pw}
\end{equation}
and $\sqrt{\frac{a^2b^2}{a^2 + b^2}}$ is monotonically nondecreasing in both $a$ and $b$ over $\R_{++}$.
\end{lemma}
\begin{proof}
Assume without loss of generality that $a = \min \{a, b\}$.
Then, we have $\sqrt{\frac{a^2b^2}{a^2 + b^2}} \ge \sqrt{\frac{a^2b^2}{2b^2}} = \frac{a}{\sqrt{2}}$, which corresponds to the lower bound.
The upper bound follows from the fact that $\frac{b^2}{a^2 + b^2} < 1$.
Monotonicity is verified by examining the partial derivatives, which are always positive.
\end{proof}
\begin{corollary}\label{cor:pyr_width_lb_for_k_polytopes}
Let $\mathcal{P} \defi \prod_{i\in[k]} \mathcal{P}_i$ be a product of polytopes. Then,
\[
\delta_\mathcal{P} \geq \frac{1}{\sqrt{2}^{\ceil{\log_2 (k+1)}}} \min_{i \in [k]} \pw{\mP_i} =\bigomegal{\frac{1}{\sqrt{k}}\min_{i\in[k]} \delta_{\mathcal{P}_i}}.
\]
\end{corollary}

\begin{proof}
We use induction and \cref{fact:cart_prod_kpolytopes} 
to encode the facial distance of the $k$-Cartesian product.
For the lower bound, the base case with $k=2$ is implied by \cref{prop:pyrwidth_mPmQ,lemma:preparatory_bounds_pw}.
For the induction step, we assume that the lower bound holds for the Cartesian product of $m \le \ceil{k/2}$ polytopes and consider the two polytopes $\mP_1 \times \cdots \times \mP_{\ceil{k/2}}$ and $\mP_{\ceil{k/2}+1} \times \cdots \times \mP_k$.
Then,
\[
\begin{array}{rcl}
\pw{1:k} &\ge& \frac{1}{\sqrt{2}}\min\left\{\pw{\mP_1\cdots\mP_{\ceil{k/2}}}, \pw{\mP_{\ceil{k/2}+1}\cdots \mP_k}\right\}\\
&\ge&
\frac{1}{\sqrt{2}} \displaystyle\min \left\{
\frac{1}{\sqrt{2}^{\ceil{\log_2 \ceil{k/2}}}}\min_{i \in \left\{1, \dots, \ceil{k/2}\right\}} \pw{\mP_i}\,,\,
\frac{1}{\sqrt{2}^{\ceil{\log_2 \ceil{k/2}}}}\min_{j \in \left\{\ceil{k/2}+1, \dots, k\right\}} \pw{\mP_j}
\right\}\\
&\ge&
\displaystyle \frac{1}{\sqrt{2}^{1 + \ceil{\log_2 \ceil{k/2}}}}
\min_{i \in [k]} \pw{\mP_i}\\
&\geq&
\displaystyle \frac{1}{\sqrt{2}^{\ceil{\log_2 (k+1)}}}
    \min_{i \in [k]} \pw{\mP_i} = \bigomegal{\frac{1}{\sqrt{k}} \min_{i \in [k]} \pw{\mP_i}},
\end{array}
\]
where the first inequality is the base case with two polytopes, the second one applies the induction step, the third one uses $\sqrt{2}^{\ceil{\log_2 \ceil{k/2}}} \ge \sqrt{2}^{\floor{\log_2 \floor{k/2}}}$ and computes the inner $\min$.

\end{proof}

\subsection{Vertex-facet distance}\label{a:vf}

In order to prove \cref{prop:vfd_PxQ_bounds}, we define the \emph{inner vertex-facet distance} of $\mP$ with respect to $F_\mP \in \facets{\mP}$ as
\[
\vf{\mP, F_\mP} \defi \dist{\aff{F_\mP}, \verts{\mP}\setminus \verts{F_\mP}}\,,
\]
so that $\vf{\mP} = \min_{F_\mP \in \facets{\mP}} \vf{\mP, F_\mP}$.
\begin{proof}\linkofproof{prop:vfd_PxQ_bounds}
It is enough to show the statement for the Cartesian product of two polytopes $\mP$ and $\mQ$.
By \cref{def:vfd},
\begin{equation}
\vf{\mP \times \mQ} = \displaystyle\min_{F \in \facets{\mP \times \mQ}} \vf{\mP \times \mQ, F} = \min_{F \in \facets{\mP \times \mQ}}\dist{\aff{F}, \verts{\mP \times \mQ}\setminus \verts{F}}\,.
\label{eq:vfd_prod}
\end{equation}
All product facets $F$ are of the form $F_\mP \times \mQ$, where $F_\mP$ is a facet of $\mP$, or $\mP \times F_\mQ$, where $F_\mQ$ is a facet of $\mQ$.
Given $F_\mP \times \mQ \in \facets{\mP \times \mQ}$, its vertices are
\begin{equation}\label{eq:verts_facet_equiv}
\begin{array}{rcl}
\verts{\mP \times \mQ} \setminus \verts{F_\mP \times \mQ} &=& \left(\verts{\mP} \times \verts{\mQ}) \setminus (\verts{F_\mP} \times \verts{\mQ}\right)\\
&=& \left(\verts{\mP} \setminus \verts{F_\mP}\right) \times \verts{\mQ}\,.
\end{array}
\end{equation}
so
\[
\begin{array}{rcl}
\vf{\mP \times \mQ, F_\mP \times \mQ} &=&
\dist{\aff{F_\mP \times \mQ}, \left(\verts{\mP} \setminus \verts{F_\mP}\right) \times \verts{\mQ}}
\\
&=& \displaystyle\min_{(x,y) \in \aff{F_\mP \times \mQ}} \min_{(z_1, z_2) \in \left(\verts{\mP} \setminus \verts{F_\mP}\right) \times \verts{\mQ}} \|(x, y) - (z_1, z_2)\|\\ &=&
\displaystyle\min_{\substack{x \in \aff{F_\mP}\\ y \in \aff{\mQ}}} \min_{\substack{z_1 \in \verts{\mP} \setminus \verts{F_\mP}\\ z_2 \in \verts{\mQ}}} \sqrt{\|x - z_1\|^2 + \|y - z_2\|^2}\\
&=& \vf{\mP, F_\mP}\,.
\end{array}
\]
The first equality is the definition of inner vertex-facet distance (see \cref{def:vfd}).
In the second equality, we apply the fact that $\aff{F \times \mQ} = \aff{F} \times \aff{\mQ}$ and \cref{eq:verts_facet_equiv}.
The third equality is just a reformulation over the individual sets in $\mP$ and in $\mQ$.
To obtain the last equality, we minimize the first summand by choosing $x, z_1$ achieving $\vf{\mP, F_\mP}$ and the second summand by selecting $y = z_2$ so that $\|y - z_2\|^2 = 0$, which we can do as $\verts{\mQ} \subset \aff{\mQ}$.
Similarly, the vertices of facets of the form $\mP \times F_\mQ$ are $\verts{\mP \times \mQ} \setminus \verts{\mP \times F_\mQ} = \verts{\mP} \times \left(\verts{\mQ} \setminus \verts{F_\mQ}\right)$, thus $\vf{\mP \times \mQ, \mP \times F_\mQ} = \vf{\mQ, F_\mQ}$.

Plugging the above equalities in \cref{eq:vfd_prod} yields:
\[
\begin{array}{rcl}
\vf{\mP \times \mQ} &=& \displaystyle\min_{F \in \facets{\mP \times \mQ}} \vf{\mP \times \mQ, F}\\
&=& \min \left\{\displaystyle \min_{F_\mP \in \facets{\mP}} \vf{\mP, F_\mP}, \min_{F_\mQ \in \facets{\mQ}} \vf{\mQ, F_\mQ}\right\}\\
&=& \min\{\vf{\mP}, \vf{\mP}\}\,,
\end{array}
\]
where the last equality follows from the definition of inner vertex-facet distance.
\end{proof}
\subsection{Properties of the objective function in \cref{prob:intersection_k}}\label{a:obj}

\begin{proof}\linkofproof{prop:lsmooth_cvx_pl_intersection_k}\label{proof:lsmooth_cvx_intersection_k}
We recall that $f(x) = \frac{1}{2k} \innp{M_k x, x}$, with $M_k \defi (k I_{k\times k} - \ones_{k} \ones_{k}^T) \otimes I_n \in \mathbb{R}^{nk \times nk}$. In the following, we use the shorthand $A_k \defi (k I_{k\times k} - \ones_{k} \ones_{k}^T)$.
As $f(x)$ is twice differentiable on $\R^{nk}$, we can compute its Hessian $\nabla^2 f(x) = \frac{1}{k}M_k$.
Moreover, $f(x)$ is a weighted sum of squared Euclidean norms with positive weights, so $\innp{M_k x, x} \geq 0$ for all $x \in \R^{nk}$. Thus, $M_k$ is positive semidefinite and $f$ is convex. 

Next, we characterize the eigenvalues of the Hessian $\frac{1}{k}M_k$.

First we note that, as $M_k$ is symmetric, and thus its eigenvalues are real. Then, since $f$ is convex and twice differentiable, the $L$-smoothness condition in \cref{def:conv_smooth} is equivalent to the eigenvalues of the Hessian being upper bounded by $L$, i.e., whether it holds that $
\frac{1}{k}\mu(M_k) \leq L$, for all $\mu(M_k) \in \R$.
As the eigenvalues of a Kronecker product are the product of the eigenvalues of the involved matrices, counting algebraic multiplicities,
the eigenvalues of the Hessian are of the form $\frac{1}{k}\mu(A_k)$, with $\mu(A_k)$ being the eigenvalues of $A_k$.
Notice that $\ones_k\ones_k^T$ is a rank-$1$ matrix, so it has exactly one nonzero eigenvalue and all other ones are zero. Exhibiting the eigenpair $(\ones_k\ones_k^T) \ones_k = k\ones_k$ shows that its single nonzero eigenvalue is its trace $k$, with multiplicity $1$.
Since the other matrix is a multiple of the identity, it follows that $\frac{1}{k}\mu(A_k)$ consists of $k$ eigenvalues with value $1$, and one that is $0$.
Therefore, the smoothness constant is $L=1$.


Now we prove the \PL{} property. Since $f(x)$ is convex and differentiable over the compact set $\mP \subset \R^{nk}$, we have that $\Omega^\ast$ is nonempty and a minimum exists.
Next, we observe that
\[
A_k^2 = (kI_k - \ones_k \ones_k^T)^2 = k^2I_k + k\ones_k\ones_k^T - 2k\ones_k\ones_k^T = k (kI_k - \ones_k\ones_k^T) = kA_k\,,
\]
so that $M_k^2 = (A_k \otimes I_n)^2 = A_k^2 \otimes I_n = k(A_k \otimes I_n) = kM_k$.
Therefore, $\|\nabla f(x)\|^2 = \frac{1}{k^2} \|M_k x\|^2 = \frac{1}{k^2}\innp{M_k^2x, x} = \frac{1}{k}\innp{M_k x, x} = 2f(x)$.
We then verify that $f$ is $1$-\PL{}, because $\frac{1}{2}\|\nabla f(x)\|^2 = f(x) \geq \left( f(x) - f^\ast \right)$ since the minimum $f^\ast$ is clearly nonnegative.
\label{proof:qg_pl_intersection_k}
\end{proof}
\subsection{Linear rates in the product in terms of the pyramidal width}\label{a:linconv_pw}

The goal of this section, and of the following \cref{a:linconv_vfd}, is to quantify linear convergence rates of the \AFW{} method for optimizing \PL{} objectives over product polytopes. In this work, we achieve this by applying the novel bounds we found on the pyramidal width and the vertex-facet distance of the product polytope, presented in \cref{prop:pyrwidth_mPmQ} and \cref{prop:vfd_PxQ_bounds}.

Linear convergence of certain \FW{} variants using Hölder error bounds (HEB), or \emph{sharpness} conditions, which generalize the \PL{} condition, is known and has been established in previous works (see, e.g., \cite{BCC+23,KAP21}). 
These results leverage a combination of the scaling inequality from \cref{thm:pyrwidth}, based on a \gabrev{positive} condition number of the polytope, and
\gabrev{an HEB instead of strong convexity.}
Our convergence results in \cref{lemma:geom_pl} and \cref{prop:linconv_pw} follow ideas similar to those in, say, {\cite[Theorem 2.29, Lemma 3.32, Corollary 3.33]{BCC+23}}. For completeness, here we include the full analysis that accounts for the specific quantities in our setting and yield slightly different convergence rates.

Given $a_t, w_t$ defined at lines 3-5 of \cref{algo:AFW}, we first recall {\cite[Theorem 3]{LJ15}}, that we write conveniently with our notation:
\begin{theorem}[{\cite[Theorem 3]{LJ15}}]\label{thm:scaling_ineq}
Let $\mP$ be a polytope, $x_t \in \mP$ be a suboptimal point and $\mA_t$ be an active set for $x_t$. Let $x^\ast$ be an optimal point and corresponding error direction $\hat{e}_t =  \frac{x_t - x^\ast}{\|x_t - x^\ast\|}$, and gradient $\nabla f(x_t)$ (and so $\innp{\nabla f(x_t), \hat{e}_t} > 0$). Let $\tilde{d}_t = a_t - w_t$ be the pairwise \FW{} direction obtained over $\mA_t$ and $\mP$, respectively, with gradient $\nabla f(x_t)$. Then
$\frac{\innp{\nabla f(x_t), \tilde{d}_t}}{\innp{\nabla f(x_t), \hat{e}_t}} \ge \delta_\mX$
\label{thm:pyrwidth}
\end{theorem}
To quantify rates of \AFW{} for solving \cref{prob:intersection_k} and, more generally, for optimizing a \PL{} objective over a product polytope, we rely on the following lemmas:
\begin{lemma}[\gabrev{{\cite[Lemma 3.32]{BCC+23}}}]
\label{lemma:geom_pl}
Let $\mP$ be a polytope with pyramidal width $\pw{\mP} > 0$ and let $f$ be a convex and $\mu$-\PL{} function. Then, it holds:
\begin{equation}
%
    \innp{\nabla f(x_t), a_t - w_t}^2 \ge \frac{\pw[2]{\mP} \mu}{2}h_t\,,
\end{equation}
with $a_t \in \arg\max_{v \in \mA_t} \innp{\nabla f(x_t), x}$, $w_t \in \arg\min_{v \in \mP} \innp{\nabla f(x_t), x}$ and $h_t \defi f(x_t) - f(x^\ast)$.
\end{lemma}
%
%
\begin{proof}
\gabrev{
In order to prove the lemma, we rely on the notion of quadratic growth (\QG{}), which is known to be equivalent to the \PL{} condition under convexity, see \cref{def:pl}.
This holds, in particular, in the case of our objective in \cref{prob:intersection_k}.
We then get the following inequalities:}
\begin{equation}
\begin{array}{rcl}
%
\innp{\nabla f(x_t), a_t - w_t}^2 &\ge & \delta_\mX^2\frac{\innp{\nabla f(x_t), x_t - x^\ast}^2}{\|x_t - x^\ast\|^2}\\
&\ge & \delta_\mX^2\frac{ h_t^2}{\|x_t - x^\ast\|^2}\\
&\ge& \delta_\mX^2\frac{ h_t}{\|x_t - x^\ast\|^2} \frac{\mu}{2}\|x_t - x^\ast\|^2\\
& = & \frac{\delta_\mX^2 \mu}{2}h_t\,,
\end{array}
\end{equation}
the first one by \cref{thm:pyrwidth} and picking $x^\ast \in \argmin_{z \in \Omega^\ast} \|x_t - z\|^2$, the second one by convexity of $f$, and the last one because $f$ is $\mu$-QG.
\end{proof}

\begin{lemma}\label{lemma:d_bound}
For each step $t$ of \cref{algo:AFW}, the following holds:
\[
2 \innp{\nabla f(x_t), d_t} \ge \innp{\nabla f(x_t), a_t - w_t},
\]
with $w_t \in \arg\min_{v \in \mP} \innp{\nabla f(x_t), x}$.
\end{lemma}
\begin{proof}
If \cref{algo:AFW} takes a \FW{} step (Line 6), $d_t = x_t - w_t$ and $\innp{\nabla f(x_t), x_t - w_t} \ge \innp{\nabla f(x_t), a_t - x_t}$, and then $\innp{\nabla f(x_t), 2(x_t - w_t)} \ge \innp{\nabla f(x_t), x_t - w_t + a_t - x_t}$.
If, instead, \cref{algo:AFW} takes an away step (Line 11), $\innp{\nabla f(x_t), a_t - x_t} \ge \innp{\nabla f(x_t), x_t - w_t}$, and then $\innp{\nabla f(x_t), 2(a_t - x_t)} \ge \innp{\nabla f(x_t), a_t - x_t + x_t - w_t}$.
\end{proof}

\begin{lemma}[{\cite[Lemma 1.5]{BCC+23}}]\label{lemma:progress_smooth_survey}
Let $f$ be an $L$-smooth function, and $y \defi x-\lambda d$, where $d$ is an arbitrary vector (i.e., a direction) and $\lambda > 0$. Then, if $y \in \mathsf{dom}f$;
\[
f(x) - f(y) \ge \lambda\frac{\innp{\nabla f(x), d}}{2} \quad \text{ for } 0 \le \lambda \le \frac{\innp{\nabla f(x), d}}{L \|d\|^2}\,.
\]
\end{lemma}

\Dav{
Having the following proof is stupid since it is implied by \citep[Corollary 3.33]{BCC+23} with $\theta = 1/2$, since for that value of $\theta$ the sharpness condition is exactly QG.
}

\begin{proof}\linkofproof{prop:linconv_pw}
    In the first part of the proof, we show linear convergence of \cref{algo:AFW} over $\mP$ under the \PL{} inequality. Our proof adapts the one of {\cite[Theorem 2.29]{BCC+23}}, by exploiting the PL condition rather than strong convexity. This is a routinary analysis and not one of the main contribution of our work. However, after finishing this analysis, we can use the values of our condition numbers in the product polytope to obtain a quantified linear convergence rate to see the impact of our study.
%

By \cref{lemma:geom_pl} and \cref{lemma:d_bound}, it holds:
\begin{equation}\label{eq:linconv_pw_1}
h_t \le \frac{2\innp{\nabla f(x_t), a_t - w_t}^2}{\pw[2]{\mP}\mu} \le \frac{8\innp{\nabla f(x_t), d_t}^2}{\pw[2]{\mP}\mu}.
\end{equation}
for $h_t \defi f(x_t)-f(x^\ast)$. For every step of \cref{algo:AFW}, by first applying convexity and then either a \FW{} step or an away step, we have
\begin{equation}\label{eq:linconv_pw_2}
h_t \le \innp{\nabla f(x_t), x_t - w_t} \le \innp{\nabla f(x_t), d_t},
\end{equation}
where we first applied convexity of $f$ and linear minimization, and then either a FW step or an away step.
Recall line 14 of \cref{algo:AFW}: $\lambda_t \defi \min\left\{ \Lambda_t,\; \frac{\langle \nabla f(x_t), d_t \rangle}{L\,\|d_t\|^2} \right\}$, it holds:
\begin{equation}\label{eq:linconv_pw_3}
\begin{array}{rclr}
f(x_t) - f(x_{t+1}) &=& h_t - h_{t+1}\\
&\ge& \frac{\innp{\nabla f(x_t), d_t}}{2} \min \left\{ \Lambda_t, \frac{\innp{\nabla f(x_t), d_t}}{L\|d_t\|^2}\right\} &[\text{\cref{lemma:progress_smooth_survey}, line 14 \cref{algo:AFW}}]\\
&\ge & \min \left\{ \Lambda_t \frac{\innp{\nabla f(x_t), d_t}}{2}, \frac{\innp{\nabla f(x_t), d_t}^2}{2L\diam[2]{\mP}}\right\}&\\
&\ge & \min \left\{\Lambda_t \frac{h_t}{2}, \frac{\pw[2]{\mP}\mu}{16L\diam[2]{\mP}}h_t\right\}\,, & [\text{\cref{eq:linconv_pw_1} and \cref{eq:linconv_pw_2}}]
\end{array}
\end{equation}
so
\begin{equation}\label{eq:linconv_pw_4}
h_{t+1} \le \left(1 - \min \left\{\frac{\Lambda_t}{2}, \frac{\pw[2]{\mP}\mu}{16L\diam[2]{\mP}}\right\}\right)h_t\,.
\end{equation}
We now proceed to show $h_{t+1} \le \left(1 - \frac{\mu\pw[2]{\mP}}{16L\diam[2]{\mP}}\right)h_t$ for any steps and at least half of the iterations, making sure that $\Lambda_t$ is bounded away from 0 to get linear rates.

Recall line 12 of \cref{algo:AFW}: $\lambda_t \defi \min\left\{ \Lambda_t,\; \frac{\langle \nabla f(x_t), d_t \rangle}{L\,\|d_t\|^2} \right\}$.
We now proceed to show that $h_{t+1} \le \left(1 - \frac{\mu\pw[2]{\mP}}{8L\diam[2]{\mP}}\right)h_t$ for both \FW{} and away steps, for at least half of the iterations.

For both types of steps, if $\lambda_t < \Lambda_t$, then the minimum in the first line of \cref{eq:linconv_pw_3} is achieved by the second term, which immediately yields the desired bound.

Next, we consider the cases where $\lambda_t = \Lambda_t$.
For \FW{} steps, \AFW{} sets $\Lambda_t = 1$ (see line 8 of \cref{algo:AFW}), hence the bound follows by applying \cref{eq:linconv_pw_4}. In fact, since $\mu \le L$ and $\pw{\mP} \le \diam{\mP}$, we have $\frac{\mu\,\pw[2]{\mP}}{8L\,\diam[2]{\mP}} \le \frac{1}{8} \le \frac{1}{2}$.
For away steps, A\FW{} sets $\Lambda_t = \frac{\gamma_{t,a_t}}{1-\gamma_{t,a_t}}$.
When $\lambda_t = \Lambda_t$, a drop step is performed (see line 24), which discards the away vertex $a_t$.
Since away vertices can only be removed if the active set $\mA_t$ is nonempty, and each iteration of \AFW{} adds at most one vertex to $\mA_t$, drop steps can only occur in at most half of the iterations.
%
$h_1 \le \frac{L\,\diam{\mP}^2}{2} = C$, see also {\cite[Theorem 2.2]{BCC+23}}.

To conclude the proof, the progress estimate appearing in the theorem statement follows directly from \cref{cor:pyr_width_lb_for_k_polytopes}.
Finally, the linear rates in the form $t = O(\log\left(\frac{1}{\epsilon}\right))$ are obtained by using  $(1-x)^r \le \exp(-rx)$, $\forall r > 0, x \in \mathbb{R}$, and solving for $t$ the last inequality below, as follows:
\[
h_t
\leq C \left(1 - \frac{\mu}{16L\diam[2]{\mP}}\alpha_{\pw{}}^2\right)^{\lceil (t-1)/2 \rceil}
\le
C \exp\left(- \frac{\mu}{16L\diam[2]{\mP}} {\left\lceil \frac{t-1}{2} \right\rceil}\right)
\le \epsilon\,.
\]
\end{proof}
\subsection{Linear rates in the product in terms of the vertex-facet distance}\label{a:linconv_vfd}

Similarly to \cref{a:linconv_pw}, here we use our bounds on the vertex-facet distance of the product polytope to quantify the linear convergence rates achieved by \AFW{} when optimizing \PL{} objectives over it.

\begin{proof}\linkofproof{prop:linconv_vf}
We first \gabrev{study}
linear convergence of AFW applied to a product polytope with a \PL{} objective. Next, we use \cref{prop:vfd_PxQ_bounds} to relate the convergence rate over $\mP$, characterized by its vertex-facet distance, to rates expressed in terms of the minimum vertex-facet distance among the polytopes in the product.

For convergence analysis, we refer to {\cite[Theorem 3.1]{BS15}}. We focus on quantifying the convergence rates by computing the parameters that their convergence bound contains, such as $\alpha_{\vf{}}$ and $\kappa$, in the case of our problem.
We show that for problem \eqref{prob:intersection_k}, and more generally for optimizing functions satisfying the PL condition over a polytope, $\kappa$ can be determined in a more straightforward way than in \cite{BS15}, where this requires solving a potentially challenging optimization problem.
\gabrev{The quantity $C$ appearing in {\cite[Theorem 3.1]{BS15}} is an upper bound on the primal gap $f(x_t) - f(x^\ast)$, see their Lemma 2.4, so we set $C$ here to the primal gap itself, but $C = \frac{L\diam[2]{\mP}}{2}$ would also work, see {\cite[Theorem 2.2]{BCC+23}}.}
The convergence rate of {\cite[Theorem 3.1]{BS15}} depends on a constant $\kappa$ used to compute $\alpha_{\vf{}}$, that is defined as the optimal value of an optimization problem through their Lemmas 2.2, 2.4 and 2.5. However, in our setting we can determine $\kappa$ directly, as {\cite[Lemma 2.5]{BS15}} essentially shows the \QG{} inequality. \gabrev{Since $f$ is $\mu$-\PL{} and $\mu$-\QG{}, see \cref{def:pl} and \cref{lemma:geom_pl}, then $\kappa \defi \frac{2}{\mu}$ for our objective $f(x)$.
To conclude the proof, the progress estimate appearing in the theorem statement follows
directly from \cref{prop:vfd_PxQ_bounds}.
}
\end{proof}
\section{Algorithms}\label{a:algos}

\gabrev{The pseudocode for the \FW{} algorithm can be found in, e.g., {\cite[Algorithm 2.1]{BCC+23}}. We then} present the pseudocode of the \AFW{} algorithm.

\begin{algorithm}[H]
\KwData{Convex and smooth function $f$, start vertex $x_0 \in \verts{\mP}$, \LMO{} over $\mP$}
Initialize active set: $\mA_0 \gets \{x_0\}$\\
Initialize active set weights: $\gamma_{0, x_0} \gets 1$\\
\For{$t = 0, \ldots, T-1$}{
$w_t \gets \argmin_{v \in \mP} \langle \nabla f(x_t), v \rangle$ \Comment*[r]{\FW{} vertex}
$a_t \gets \argmax_{v \in \mA_t} \langle \nabla f(x_t), v \rangle$ \Comment*[r]{away vertex in $\mA_t$}
\If{$\langle \nabla f(x_t), x_t - w_t \rangle \geq \langle \nabla f(x_t), a_t - x_t \rangle$}{

    $d_t \gets x_t - w_t$ \Comment*[r]{\FW{} step}
    $\Lambda_t \gets 1$\\
}
\Else{
    $d_t \gets a_t - x_t$ \Comment*[r]{away step}
    $\Lambda_t \gets \frac{\gamma_{t, a_t}}{1-\gamma_{t, a_t}}$\\
}
$\lambda_t \gets \min\left\{\Lambda_t, \frac{\innp{\nabla f(x_t), d_t}}{L\|d_t\|^2}\right\}$\\
$x_{t+1} \gets x_t - \lambda_td_t$\\
\If{$\langle \nabla f(x_t), x_t - w_t \rangle \geq \langle \nabla f(x_t), a_t - x_t \rangle$}{
    $\gamma_{t+1, v} \gets (1-\lambda_t)\gamma_{t, v}\quad \forall v \in \mA_t \setminus \{w_t\}$ \Comment*[r]{perform $\mathcal{R}$: update $\mA_t$ and weights}
    $\gamma_{t+1, w_t} \gets \begin{cases}
        \lambda_t & \text{if $w_t \in \mA_t$}\\
        (1-\lambda_t)\gamma_{t, w_t} + \lambda_t & \text{otherwise}
    \end{cases}$\\
    $\mA_{t+1} \gets \begin{cases}
        \mA_t \cup \{w_t\} & \text{if $\lambda_t<1$}\\
        \{w_t\} & \text{if $\lambda_t=1$}
    \end{cases}$
}
\Else{
    $\gamma_{t+1, v} \gets (1+\lambda_t)\gamma_{t, v}\quad \forall v \in \mA_t \setminus \{a_t\}$\\
    $\gamma_{t+1, a_t} \gets (1+\lambda_t)\gamma_{t, a_t} - \lambda_t$\\
    $\mA_{t+1} \gets \begin{cases}
        \mA_t \setminus \{a_t\} & \text{if $\gamma_{t+1, a_t} = 0$} \Comment*[r]{drop step}\\
        \mA_t & \text{if $\gamma_{t+1, a_t} > 0$}
    \end{cases}$
}
}
\caption{Away Frank Wolfe {\cite[Algorithm 2.2]{BCC+23}}}
\label{algo:AFW}
\end{algorithm}
The \emph{active set} $\mA_t$ of a polytope $\mP$ at iteration $t$ is defined as follows:
\[
\mA_t \defi \{v \in \verts{\mP}\,:\, x_t = \sum_{v \in \verts{\mP}} \gamma_v v,\, \gamma>0,\, \sum_{v \in \verts{\mP}}\gamma_v = 1\}\,,
\]
and we use $\gamma_{t,v}$ to denote the active set weight of a vertex $v \in \mA_t$ to represent $x_t$.

An essential component of any \FW{} variant is the choice of step size.
One rule is the \emph{line search}, in which at iteration $t$ one approximately or exactly finds $\argmin_{\lambda \in [0,\Lambda_t]} f(x_t + \lambda(v_t - x_t))$, where $v_t$ is the \LMO{} vertex and $\Lambda_t$ maintains feasibility.
Although this rule maximally decreases the current objective, it requires solving an additional subproblem and can be expensive.
An alternative, appearing on line 14 of \cref{algo:AFW}, is the \emph{short step} rule, which is obtained by minimizing a quadratic upper bound on the objective guaranteed by its smoothness. In fact, the short step rule is a cheaper approximation of the line search one, although it comes at the expense of knowing or estimating the smoothness constant.
Both step size rules ensure monotonic decrease of $f$ \cite{BCC+23}, leading to linear rates under standard assumptions.
In fact, any convergence rate proved for the short step variant carries over to the line search version, which makes larger progress per iteration, while the short step remains the preferred choice for clean theoretical analysis due to its simple form.

We note that \cref{algo:AFW} incorporates a procedure $\mathcal{R}$, described at lines 16-25, which refines the active set $\mA_t$ and the active weight vector $\lambda_t$ at each iteration. In particular, after line 21 of the algorithm, $\mathcal{R}$ updates the representation of the current iterate $x_t$ to a more efficient one, using a reduced active set $\tilde{\mA}_t \subseteq \mA_t$ with $|\tilde{\mA}_t| \leq |\verts{\mP}|$ as in \cref{algo:AFW}. We note that $\mathcal{R}$ can also be designed to implement the Carathéodory theorem, in which case $|\tilde{\mA}_t| \le n+1$, where $n$ is the dimension of $\mP$, cf. \cite[Appendix A]{BS15}, \cite{BPW24} and references therein. This variant is advantageous when the number of vertices of $\mP$ is significantly larger than 
its dimension. The notation $N$ employed in the proof of \cref{prop:linconv_vf} refers to these two upper bounds on $|\tilde{\mA}_t|$, i.e., since the product polytope has dimension $nk$, $N \in \{|\verts{\mP}|, nk+1\}$.

We note that one can modify the analysis of the \AFW{} algorithm to obtain similar convergence rates for a primal-dual computable gap, similarly as in \citep[Appendix D.1.2]{martinez2025beyond}, which shows the linear convergence for \AFW{} exploiting the pyramidal width and strong convexity of the objective.

Next, we provide the pseudocode for the alternating linear minimization (ALM) algorithm proposed in \cite{BPW23}, an adaptation of the cyclic block-coordinate FW (C-BC-FW) algorithm from \cite{BPS15} to the set-intersection problem. In our computational experiments in \cref{s:experiments}, we use ALM as a benchmark in solving problem \cref{prob:intersection_k}.



\begin{algorithm}[H]
\KwData{$L$-smooth and convex $f$, $(x_0, y_0) \in \mP \times \mQ$, $\lambda_t, \gamma_t \in [0,1]$, $t = 0, \dots, T-1$}

\For{$t = 0$ \KwTo $T-1$}{
    $w_t \gets \arg\min_{x \in \mP} \innp{\nabla_x f(x_t, y_t), x}$

	$x_{t+1} \gets x_t - \lambda_t (x_t - w_t)$

	$z_t \gets \arg\min_{y \in \mQ} \innp{\nabla_x f(x_{t+1}, y_t), y}$

	$y_{t+1} \gets x_t - \gamma_t (y_t - z_t)$
}
\caption{Alternating Linear Minimization}
\label{algo:alm}
\end{algorithm}


%
%
%
%
%
%
%
%
%
%
%
%
%
%
%
%
%
%
%
%
\section{More Experiments}\label{app:extra_experiments}

\gabrev{In this section, we present additional plots to further validate the practical behavior of the \FW{} C-BC-\FW{} and \AFW{} algorithms.
As in \cref{s:experiments}, the plots below display the primal and \FW{} gap of the examined algorithms, plotted against the elapsed computational time, in seconds.
The software and hardware settings are the same as the ones used to obtain the figures in \cref{s:experiments}.}

\gabrev{The plots presented below show the outcome of running the experiments of \cref{s:experiments} several times but with different configurations and seeds.
Doing this, we obtain the same relative performance among the algorithms:  when solving both instances where polytopes do not intersect (\cref{fig:k2n2500_ni}, \cref{fig:k2n5000_ni}, \cref{fig:k5n5000_ni}, \cref{fig:k5n10000_ni}) and when they have a non-empty intersection (\cref{fig:k2n10000_i}, \cref{fig:k5n10000_i}), the curves again exhibit the predicted linear rate for \AFW{} on product polytopes, and confirm that its empirical advantage persists across different randomizations.
Furthermore, these figures confirm that the runs presented in \cref{s:experiments} are representative of algorithmic behavior and not statistical outliers.}
\begin{figure}[htbp!]
\centering
\includegraphics[width=\textwidth]{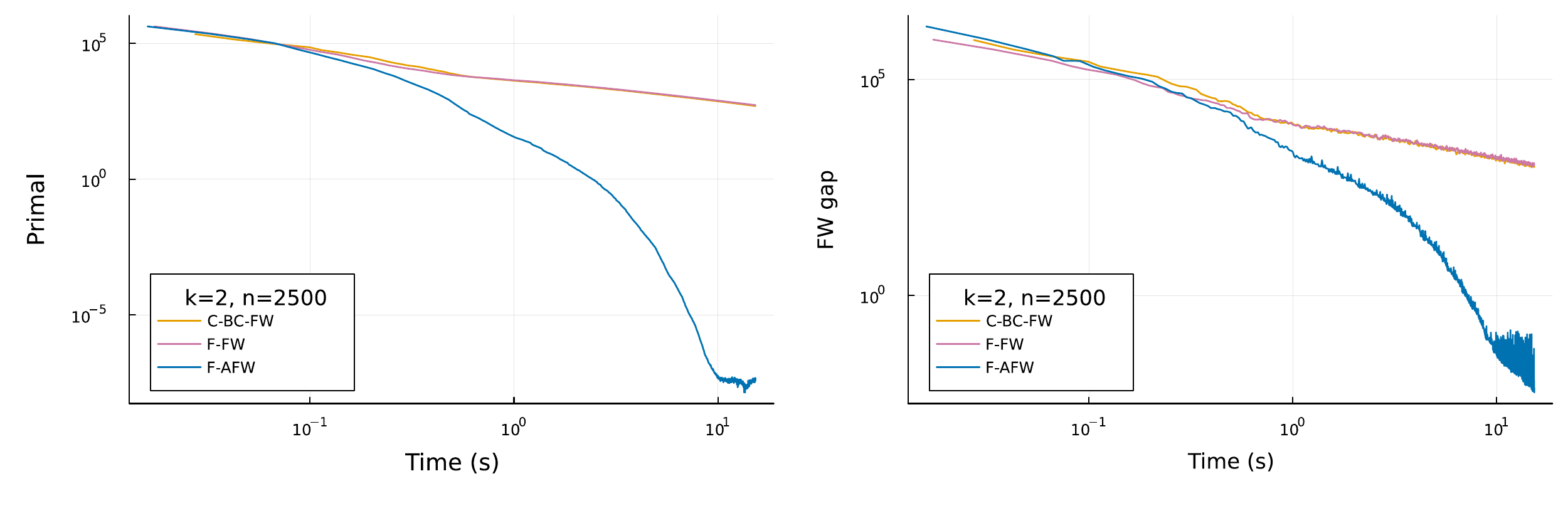}
\caption{Different variants of the AFW algorithm, run for up to $1\,000$ iterations on two non‑intersecting polytopes in $\R^{2\,500}$. }
\label{fig:k2n2500_ni}
\end{figure}
\begin{figure}[htbp!]
\centering
\includegraphics[width=\textwidth]{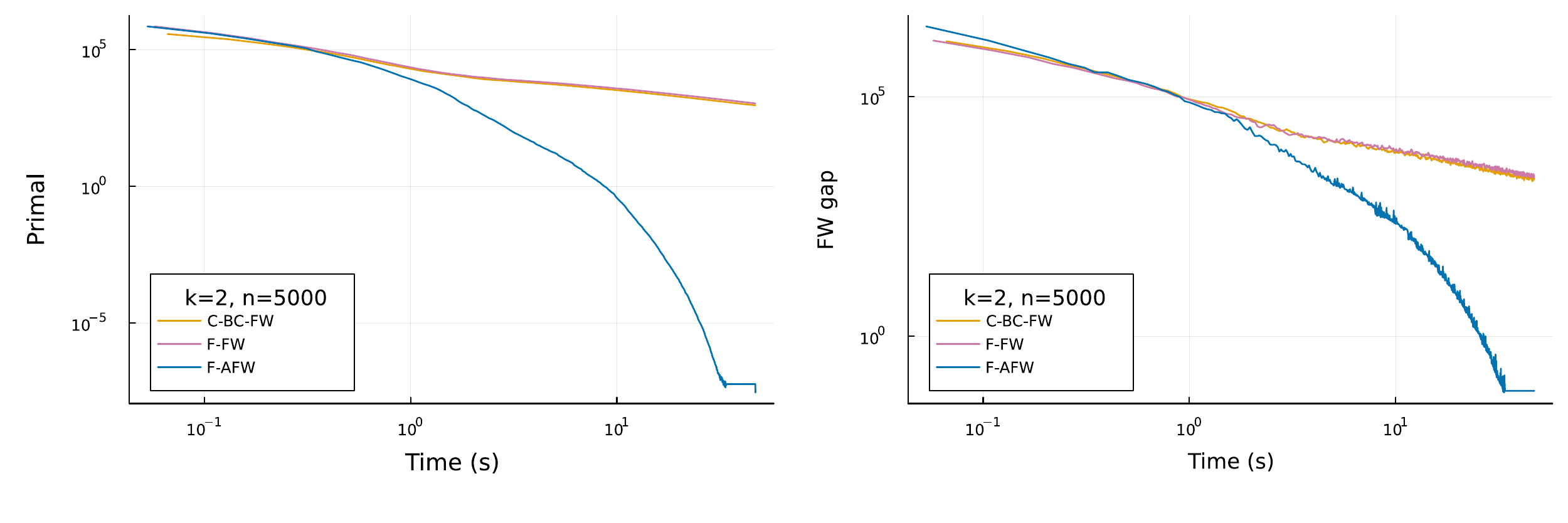}
\caption{Different variants of the AFW algorithm, run for up to $1\,000$ iterations on two non‑intersecting polytopes in $\R^{5000}$. }
\label{fig:k2n5000_ni}
\end{figure}
\begin{figure}[htbp!]
\centering
\includegraphics[width=\textwidth]{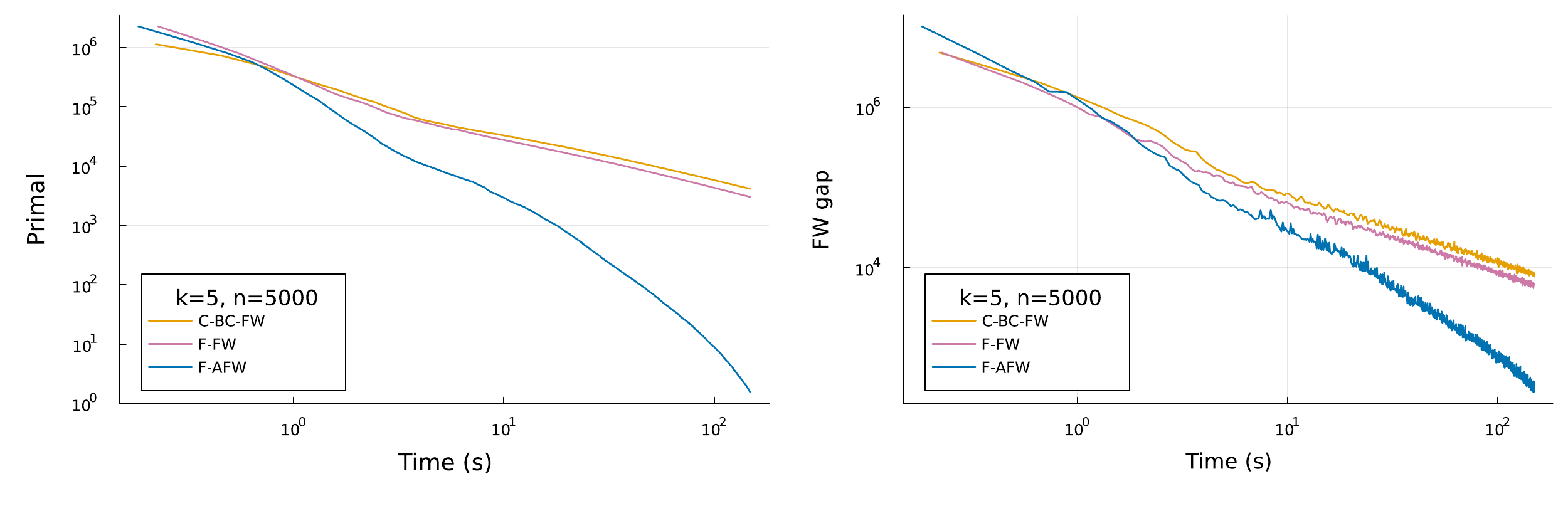}
\caption{Different variants of the AFW algorithm, run for up to $1\,000$ iterations on five non‑intersecting polytopes in $\R^{5\,000}$. }
\label{fig:k5n5000_ni}
\end{figure}
\begin{figure}[htbp!]
\centering
\includegraphics[width=\textwidth]{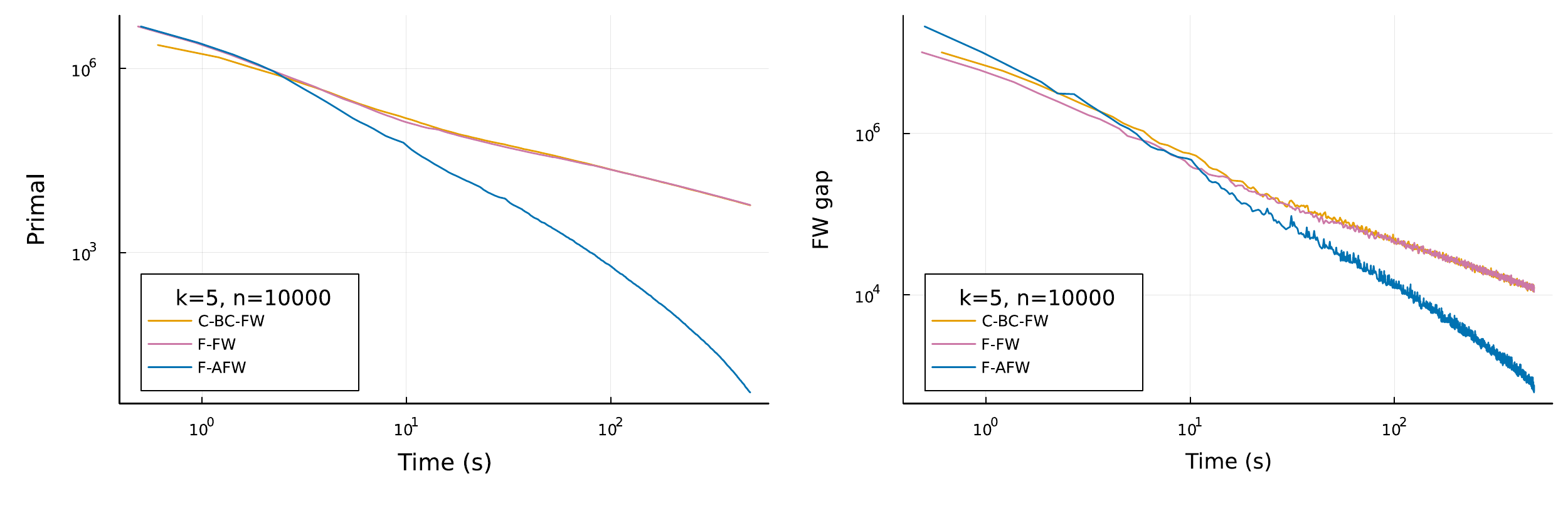}
\caption{Different variants of the AFW algorithm, run for up to $1\,000$ iterations on five non‑intersecting polytopes in $\R^{10\,000}$. }
\label{fig:k5n10000_ni}
\end{figure}
\begin{figure}[htbp!]
\centering
\includegraphics[width=\textwidth]{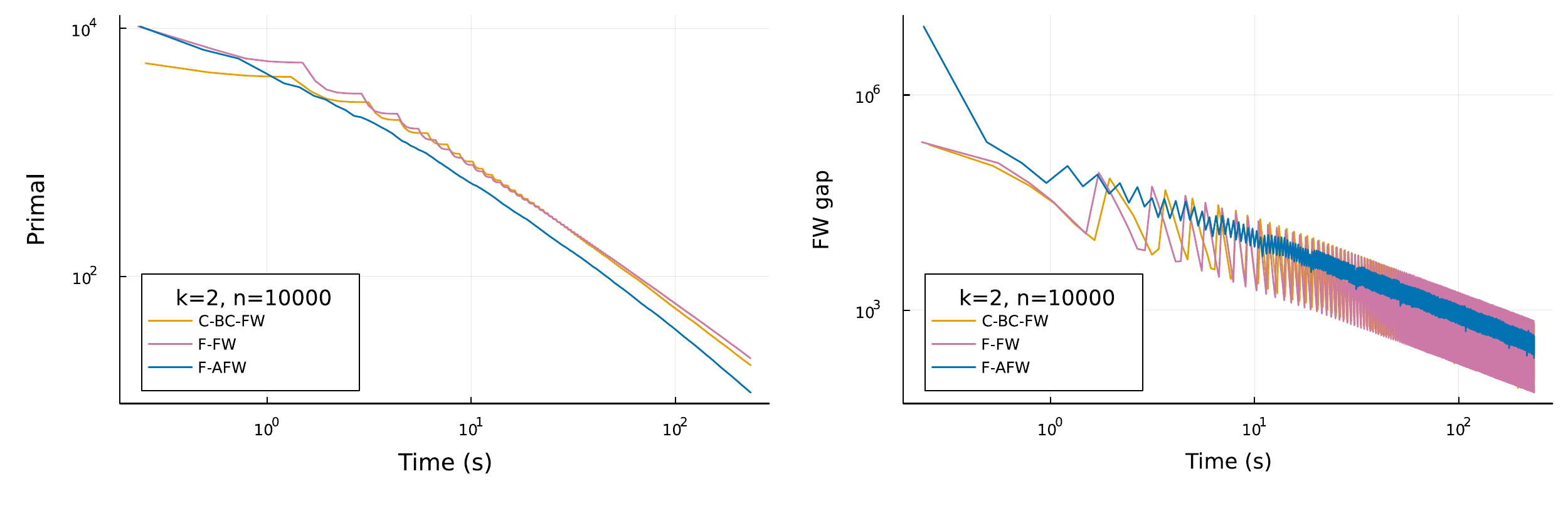}
\caption{Different variants of the AFW algorithm, run for up to $1\,000$ iterations on two intersecting polytopes in $\R^{10\,000}$. }
\label{fig:k2n10000_i}
\end{figure}
\begin{figure}[htbp!]
\centering
\includegraphics[width=\textwidth]{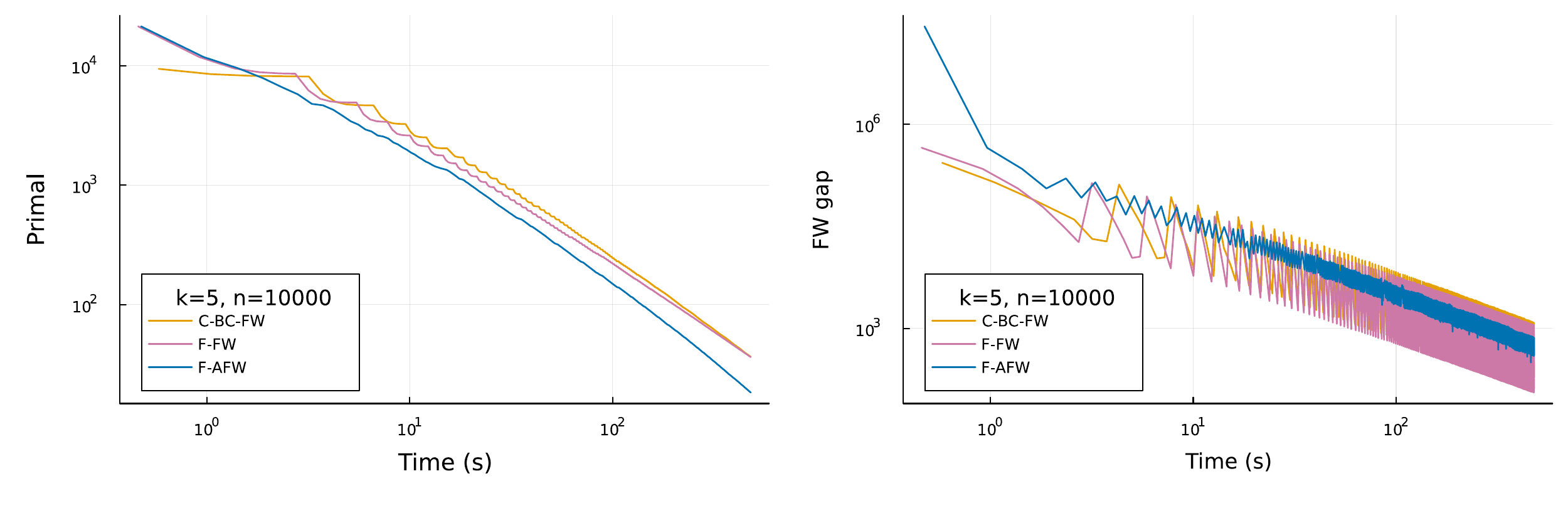}
\caption{Different variants of the AFW algorithm, run for up to $1\,000$ iterations on five intersecting polytopes in $\R^{10\,000}$. }
\label{fig:k5n10000_i}
\end{figure}

\gabrev{We also increase the scale by solving the intersection problem on two disjoint polytopes, generated as the convex hull of $100\,000$ vertices each, namely, around an order of magnitude larger than the example in \cref{fig:k2n20000}.
We provide the results of this test in \cref{fig:k2n10000_100000vertices_ni}. As shown in the plot, \AFW{} continues to outperform the other methods, underscoring its robustness even when individual \LMO{} calls become substantially more expensive. In fact, we note that, for each polytope, we always compute \LMO{}s that perform a single linear pass over the vertex list that was used to generate it, evaluating one dot product per vertex. The same implementation was used in \cref{s:experiments}. Because of this, the time to convergence in \cref{fig:k2n10000_100000vertices_ni} is longer, in absolute terms, than on instances created following the procedure explained in \cref{s:experiments}, i.e., with polytopes generated from a list of between $n$ and $2n$ vertices, see \cref{fig:k2n10000_15000vertices_ni} for a comparison with one such instance. Clearly, however, the relative performance of all the tested algorithms is virtually unchanged.}

\begin{figure}[htbp]
\centering
\includegraphics[width=\textwidth]{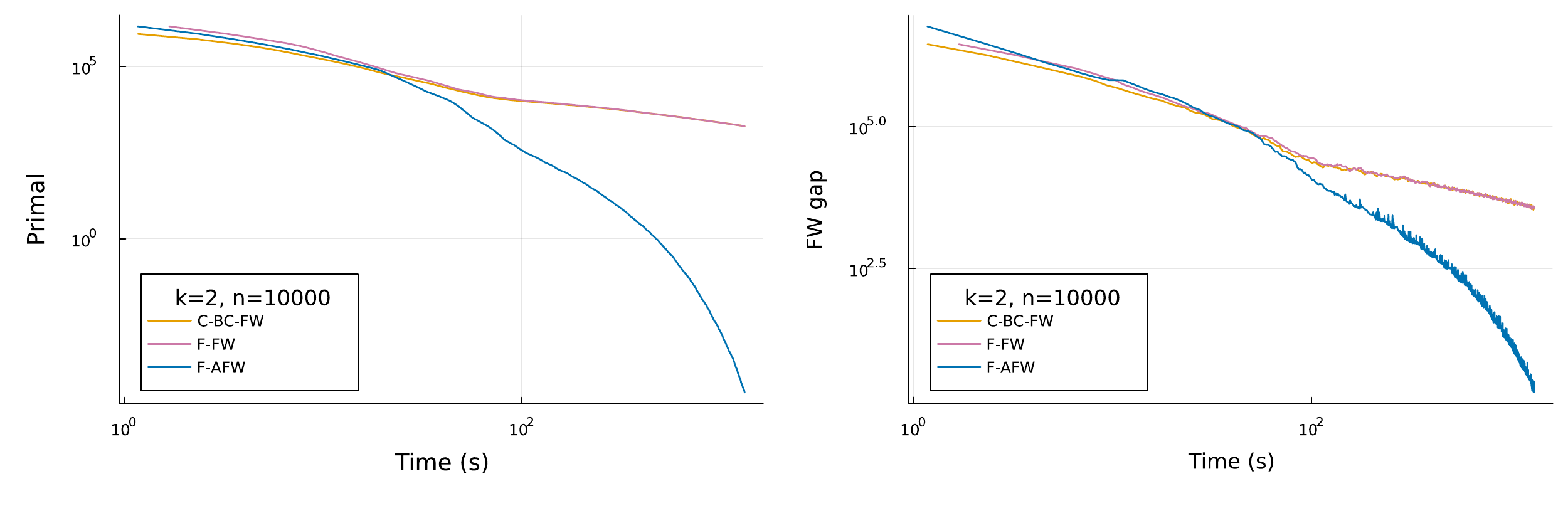}
\caption{Different variants of the AFW algorithm run for up to $1\,000$ iterations on two non‑intersecting polytopes in $\R^{10\,000}$, generated as the convex hull of $100\,000$ vertices.}
\label{fig:k2n10000_100000vertices_ni}
\end{figure}

\begin{figure}[htbp]
\centering
\includegraphics[width=\textwidth]{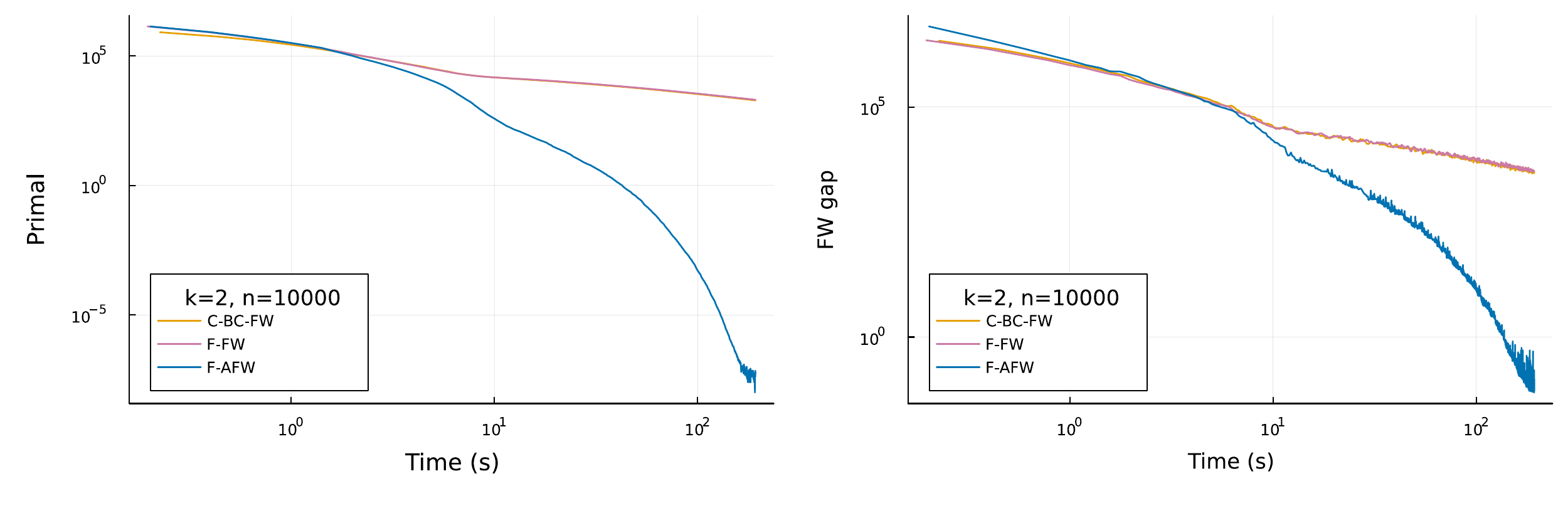}
\caption{Different variants of the AFW algorithm run for up to $1\,000$ iterations on two non‑intersecting polytopes in $\R^{10\,000}$, generated as the convex hull of $15\,000$ vertices.}
\label{fig:k2n10000_15000vertices_ni}
\end{figure}
\end{document}